\documentclass[a4paper,12pt]	{amsart}
\usepackage[T1]				{fontenc} 		
\usepackage[ansinew]			{inputenc} 		
\usepackage					{ellipsis} 		
\usepackage					{lmodern} 		
\usepackage					{amstext}
\usepackage					{amssymb}
\usepackage					{amsmath}
\usepackage					{amsthm}	
\usepackage					{bbm}			
\usepackage[fixamsmath]		{mathtools} 		
\usepackage[shortlabels]		{enumitem} 		
\usepackage					{geometry}
\usepackage[backend=biber,style=alphabetic]		{biblatex}

\usepackage					{csquotes}		
\usepackage					{graphicx, caption}
\usepackage 					{hyperref}
 
\allowdisplaybreaks
\theoremstyle{plain}
\newtheorem{thm}			{Theorem}[section]
\newtheorem{lemma}		[thm]	{Lemma}
\newtheorem{cor}			[thm]	{Corollary}
\newtheorem{proposition}	[thm]	{Proposition}

\theoremstyle{definition}
\newtheorem{definition}	[thm]	{Definition}
\newtheorem*{defi*}				{Definition}

\theoremstyle{remark}

\newtheorem{rem}		[thm]		{Remark}
\newtheorem*{exa*}		    		{Example}
\newtheorem*{rem*}					{Remark}

\DeclareMathOperator	{\IN}		{\mathbb{N}} 	\DeclareMathOperator	{\IZ}		{\mathbb{Z}}
		\DeclareMathOperator	{\IR}		{\mathbb{R}}
		\DeclareMathOperator	{\IE}		{\mathbb{E}} 
		\DeclareMathOperator	{\IP}		{\mathbb{P}}

\DeclareMathOperator	{\Var}		{Var}
\DeclareMathOperator	{\Ent}		{Ent}
\DeclareMathOperator	{\Cov}		{Cov}

\DeclarePairedDelimiter	\abs		{\lvert}{\rvert}
\DeclarePairedDelimiter	\norm		{\lVert}{\rVert}
\DeclarePairedDelimiter	\skal		{\langle}{\rangle}

\renewcommand 		{\epsilon}	{\varepsilon}
\renewcommand		{\phi}		{\varphi}
\renewcommand		{\partial}	{\mathfrak{d}}
\renewcommand		{\tilde}		{\widetilde}
\newcommand			{\fh}		{\mathfrak{h}}

\numberwithin		{equation}{section}

\title{Higher Order Concentration for Functions of Weakly Dependent Random Variables}
\author{F. G\"otze}
\address{Friedrich G\"otze, Faculty of Mathematics, Bielefeld University, Bielefeld, Germany}
\email{goetze@math.uni-bielefeld.de}
\author{H. Sambale}
\address{Holger Sambale, Faculty of Mathematics, Bielefeld University, Bielefeld, Germany}
\email{hsambale@math.uni-bielefeld.de}
\author{A. Sinulis}
\address{Arthur Sinulis, Faculty of Mathematics, Bielefeld University, Bielefeld, Germany}
\email{asinulis@math.uni-bielefeld.de}

\begin{document}
\thanks{This research was supported by CRC 1283.}
\begin{abstract}
	We extend recent higher order concentration results in the discrete setting to include functions of possibly dependent variables whose distribution (on the product space) satisfies a logarithmic Sobolev inequality with respect to a difference operator that arises from Gibbs sampler type dynamics. Examples of such random variables include the Ising model on a graph with $n$ sites with general, but weak interactions, i.e. in the Dobrushin uniqueness regime, for which we prove concentration results of homogeneous polynomials, as well as random permutations, and slices of the hypercube with dynamics given by either the Bernoulli-Laplace or the symmetric simple exclusion processes.
\end{abstract}
\date{\today}

\maketitle
\section{Introduction}	\label{section:introduction}
In this article, we study \emph{higher order} versions of the \emph{concentration of measure phenomenon} for functions of random variables $X_1, \ldots, X_n$ defined on some probability space $(\Omega, \mathcal{A}, \mathbb{P})$ with values in some Polish space $X_i: \Omega \to S_i$ which are not necessarily independent. The term \emph{higher order} shall emphasize that we prove tail estimates for functions with possibly non-bounded first order differences, or functions for which the $L^\infty$ norm of its differences increases with the size of the system, even after a proper normalization, such as quadratic forms in weakly dependent variables. \par 
To formalize this intuition we consider certain \emph{difference operators}. By a difference operator we mean an operator $\Gamma$ on the space $L^\infty(\mu)$ for some probability measure $\mu$ satisfying $\Gamma(af+b) = \abs{a} \Gamma(f)$ for $b \in \IR$ and either $a > 0$ or $a \in \IR$. The restriction $f \in L^\infty(\mu)$ is merely a minimal requirement, since $f \in L^2(\mu)$ will sometimes be sufficient to define certain operators, and in the cases that we will consider in the applications (i.e. finite probability spaces) $L^\infty(\mu)$ is the space of all functions. Hence we shall stick to this simplifying assumption. In our cases, $\mu$ is the distribution of $X \coloneqq (X_1, \ldots, X_n)$ on $S \coloneqq \times_{i = 1}^n S_i$. \par
The difference operators $\partial, \mathfrak{h}$ will be Euclidean norms corresponding to vectors $\mathfrak{h} = (\mathfrak{h}_I)_{I \in \mathcal{I}}$ or $\partial = (\partial_I)_{I \in \mathcal{I}}$ arising from the disintegration theorem on Polish spaces and can be thought of as $L^2$ and $L^\infty$ norms respectively conditioned on certain variables $\mathcal{I} \subset \mathcal{P}(\{1,\ldots,n\})$. We postpone the exact definition to Definition \ref{FirstOrderDiff} in section \ref{section:HigherOrderDiff}. \par
Using $\mathfrak{h}$, it is possible to define \emph{higher order difference operators} $\mathfrak{h}_{I_1 \ldots I_d}$ for any $d \in \mathbb{N}$ by iteration, i.e. by setting \pagebreak[3]
\begin{align} 		\label{HigherOrderDiff} 
\begin{split}	
\mathfrak{h}_{I_1 \ldots I_d}f = \mathfrak{h}_{I_1}(\mathfrak{h}_{I_2\ldots I_d}f), 
\end{split}
\end{align}
and tensors of $d$-th order differences $\mathfrak{h}^{(d)}f(x)$ with coordinates $\mathfrak{h}_{I_1\ldots I_d}f(x)$. The tensor may be regarded as a vector indexed by $\mathcal{I}^d$ and we define the norm $\abs{\mathfrak{h}^{(d)}f(x)}$ as its Euclidean norm. For instance, in certain cases $\abs{\mathfrak{h}^{(1)}f(x)}$ is just the Euclidean norm of the ``gradient'' $\mathfrak{h} f(x)$, and $\abs{\mathfrak{h}^{(2)}f(x)}$ is the Hilbert--Schmidt norm of the ``Hessian'' $\mathfrak{h}^{(2)}f(x)$. 
Additionally, we will use the notation $\norm{f}_p$ for the $p$-norm of a function $f$ (with respect to a measure $\mu$ which is clear from the context) and write
\begin{equation}
	\lVert \mathfrak{h}^{(d)}f \rVert_{p} = \lVert |\mathfrak{h}^{(d)}f| \rVert_p  \equiv \left( \mathbb{E}_\mu \abs{\mathfrak{h}^{(d)}f}^p \right)^{1/p}
\end{equation}
for any $p \in (0, \infty]$, where for $p = \infty$ this is the essential supremum with respect to $\mu$. \par

Next let us recall the notion of \emph{Poincar{\'e}} and \emph{logarithmic Sobolev inequalities} in the framework of difference operators. We say that the measure $\mu$ satisfies a \emph{Poincar{\'e} inequality} with constant $\sigma^2 > 0$ with respect to some difference operator
$\Gamma$ (in short: $\text{PI}_\Gamma(\sigma^2)$) if for all $f \in L^\infty(\mu)$
\begin{equation}	\label{PI}
	\mathrm{Var}_\mu (f) \le \sigma^2 \IE_\mu |\Gamma f|^2,
\end{equation}
where $\text{Var}_\mu (f) = \mathbb{E}_\mu f^2 - (\mathbb{E}_\mu f)^2$ is the variance functional with respect to $\mu$. \par 
Moreover, $\mu$ satisfies a \emph{logarithmic Sobolev inequality} with constant $\sigma^2 > 0$ with respect to some difference operator $\Gamma$ (in short: $\text{LSI}_\Gamma(\sigma^2)$) if for all $f \in L^\infty(\mu)$
\begin{equation}	\label{LSI}
	\Ent_\mu (f^2) \le 2 \sigma^2 \IE_\mu |\Gamma f|^2,
\end{equation}
where for any function $f \ge 0$ we denote by $\Ent_\mu (f) \coloneqq \Ent (f) \coloneqq \mathbb{E}_\mu f \log f - \mathbb{E}_\mu f \log\mathbb{E}_\mu f \in [0, \infty]$ the entropy functional with respect to $\mu$. \par 
It is well known that logarithmic Sobolev inequalities are stronger than Poincar{\'e} inequalities, i.e. if $\mu$ satisfies a logarithmic Sobolev inequality with constant $\sigma^2$, it also satisfies a Poincar{\'e} inequality with the same constant $\sigma^2$, see for example \cite{AS94} in the context of Markov semigroups, \cite[Lemma 3.1]{DSC96} in the framework of Markov chains, or \cite[Proposition 3.6]{BT06}, where also modified logarithmic Sobolev inequalities have been considered. We shall tacitly use this implication. \par
We formulate a general result in section \ref{section:GeneralResult} which may be applied to functions of the spins in Ising models, of random permutations and on slices of the hypercube. We start with an application to the Ising model with general interactions.
\subsection{Ising model}
In the special case of the Ising model $q^n$ on $n$ sites the difference operator under consideration can be written as 
\[
	\abs{\partial f}^2(\sigma) = \frac{1}{2} \sum_{i = 1}^n (f(\sigma) - f(T_i \sigma))^2 q^n(-\sigma_i \mid \sigma_1, \ldots, \sigma_{i-1}, \sigma_{i+1}, \ldots, \sigma_n),
\]
where $T_i \sigma = (\sigma_1, \ldots, \sigma_{i-1}, -\sigma_i, \sigma_{i+1}, \ldots, \sigma_n)$ is the switch operator of the $i$-th spin and $q^n ( \cdot \mid \sigma_1, \ldots, \sigma_{i-1}, \sigma_{i+1}, \ldots, \sigma_n)$ is the conditional measure. We call this the \emph{difference operator of the Gibbs sampler} (or Glauber dynamics). Additionally, we have
\[
	\abs{\mathfrak{h}f}^2(\sigma) = \frac{1}{2} \sum_{i = 1}^n (f(\sigma) - f(T_i\sigma))^2.
\]

\begin{proposition}		\label{Ising:Prop:LSI}
	Let $q^n$ be the probability measure on $\{-1,+1\}^n$ defined by normalizing $\pi(\sigma) = \exp\left( \frac{1}{2} \sum_{i,j} J_{ij} \sigma_i \sigma_j + \sum_{i = 1}^n h_i \sigma_i \right)$, where $\norm{h}_\infty \le \tilde{\alpha}$ and $J = (J_{ij})_{i,j}$ satisfies $J_{ii} = 0$ and 
	\begin{equation}	\label{Ising:eqn:Jnorm}
	\norm{J}_{1 \to 1} = \max_{i = 1,\ldots,n} \sum_{j = 1}^n \abs{J_{ij}} \le 1-\alpha.
	\end{equation}
	There is a constant $C = C(\alpha, \tilde{\alpha})$ depending only on $\alpha$ and $\tilde{\alpha}$ such that for the difference operator of the Gibbs sampler given above we have
	\begin{equation}	\label{Ising:eqn:LSI}
		\Ent_{q^n}(f^2) \le 2C \IE_{q^n} \abs{\mathfrak{d}f}^2.
	\end{equation}
	Moreover, for any $f: \{ -1, +1\}^n \to \IR$ we have
	\begin{equation}	\label{Ising:eqn:momentinequality}
		\norm{f}_p^2 - \norm{f}_2^2 \le 2C(p-2) \norm{\mathfrak{h} f}_p^2.
	\end{equation}
\end{proposition}

\begin{rem*}
	This can be seen as a generalization of the logarithmic Sobolev inequality on $\{ -1, +1\}^n$ equipped with the uniform measure, which corresponds to the Ising model without any interactions and without an external field, i.e. $J = 0$ and $h = 0$. In general the case $J = 0$ yields $n$ independent random variables $\sigma_1,\ldots, \sigma_n$ with $\IP(\sigma_i = 1) = \frac{1}{2}(1 + \tanh(h_i))$. Thus a uniform bound on $\norm{h}_\infty$ is necessary in order for the logarithmic Sobolev constant to be stable, see e.g. \cite[Theorem A.1]{DSC96}. \par 
	Condition \eqref{Ising:eqn:Jnorm} appears in various contexts, we shall call it Dobrushin uniqueness condition, see for example \cite{Ku03}, equations (2.1) and (2.2). The Dobrushin uniqueness condition implies that the coupling matrix $A$ of the Ising model satisfies $\norm{A}_{2 \to 2} \le 1-\alpha$, which is a requirement to apply an approximate tensorization result.\par 
	In a series of papers \cite{Z92, SZ92a, SZ92b} B. Zegarlinski and D. W. Stroock have established the equivalence of the logarithmic Sobolev inequality and the so-called Dobrushin-Shlosman mixing condition on $\{-1, +1\}^{\IZ^d}$. Here we prove one implication using an approximate tensorization result by K. Marton \cite{Ma15} for the easier case $\{-1,+1\}^n$.
\end{rem*}

From an iteration procedure we obtain the following Theorem establishing tail estimates for functions of spins in the Ising model with bounded differences of higher order.

\begin{thm}	\label{Ising:Thm:d-th-order-conc}
	Let $d \in \mathbb{N}$, $q^n$ as in Proposition \ref{Ising:Prop:LSI} and $f$ be any function. Assuming the conditions
	\begin{equation}	\label{Ising:eqn:d-th-order-conc-cond1}
	\lVert \mathfrak{h}^{(k)} f \rVert_{2} \le 1 \quad \text{ for all } k = 1, \ldots, d-1
	\end{equation}
	and
	\begin{equation}	\label{Ising:eqn:d-th-order-conc-cond2}
	\norm{ \mathfrak{h}^{(d)} f }_{\infty} \le 1,
	\end{equation}
 	there exists some constant $C = C(\alpha, \tilde{\alpha}, d) > 0$ such that
	\[
		\mathbb{E}_{q^n} \exp\left(C |f - \IE_{q^n} f|^{2/d}\right) \le 2.
	\]
	Especially we have
	\[
		q^{n}(\abs{f - \IE_{q^n} f} \ge t) \le  2 \exp\left( - C t^{2/d} \right).
	\]
\end{thm}

As an application, one can show concentration results for homogeneous polynomials of spins in the Ising model with bounded coefficients as follows. To begin with, let us consider the case of an Ising model without external field.

\begin{thm} 	\label{Ising:Thm:ConcPolynomials}
	Let $d \in \mathbb{N}$, $q^n$ be an Ising model as in Proposition \ref{Ising:Prop:LSI} with $h = 0$. There is a constant $c = c(d, \alpha) > 0$ such that for any $d$-tensor $a = (a_I)_{\abs{I} = d}$ the $d$-homogeneous polynomial $f = \sum_{\abs{I} = d} a_I \prod_{i \in I} \sigma_i \eqqcolon \sum_{\abs{I} = d} a_I \sigma_I$ satisfies for all $t > 0$
	\begin{equation}	\label{Ising:eqn:ConcPolynomials}
		q^n(\abs{f - \IE_{q^n} f} \ge t) \le 2 \exp \left( - \frac{t^{2/d}}{cn \norm{a}_\infty^{2/d}} \right).
	\end{equation}
\end{thm}

Note that by homogeneity we could impose without loss of generality the condition $\sup_{\abs{I} = d} \abs{a_I} \le 1$ and remove $\norm{a}_\infty^{2/d}$ in the exponentiation, since a simple rescaling yields for any function $f = \sum_{\abs{I} = d} a_I \sigma_I$
	\begin{equation}
		q^n(\abs{f - \IE_{q^n} f} \ge t) \le 2 \exp \left( - \frac{t^{2/d}}{cn\norm{a}_\infty^{2/d}} \right).
	\end{equation}

This result improves upon \cite[Theorem 1]{GLP17} as well as on \cite[Theorem 5]{DDK17} by removing all logarithmic dependencies in the window of concentration and in the concentration parameter in the exponential. This bound is optimal in terms of the dependence on $t$ and $n$, since the uniform measure $\mu = \otimes_{i = 1}^n \frac{1}{2}(\delta_{-1} + \delta_{+1})$ can also be interpreted as an Ising model and via hypercontractivity arguments and the Fourier-Walsh decomposition one can see that 
\[
	\mu(\abs{f - \IE_\mu f} \ge t) \le 2 \exp\left( -\frac{t^{2/d}}{C(d)n} \right)
\]
for a $d$-homogeneous polynomial $f$, see for example \cite[Chapter 5.3]{BGL13} or \cite[Chapter 3.8]{DDK17}. \par
For $d \in \{1,2,3,4\}$ we also provide more accurate estimates for $f = \sum_{\abs{I} = d} a_I \sigma_I$ using the Hilbert-Schmidt norms of the tensor $a = (a_I)_{\abs{I} = d}$ by approximating $f$ by a lower-order polynomial, i.e. we will see that for some constants $C_1 = C_1(d,\alpha) > 0, C_2 = C_2(d, \alpha) > 0$
\[
	q^n(\abs{f - \IE_{q^n} f} \ge t) \le C_1 \exp\left( - \frac{t^{2/d}}{C_2 \norm{a}_{HS}^{2/d}} \right). 
\]
However it will become clear that this approach is cumbersome, since one needs to consider an approximation by a $(d-1)$-th order polynomial and keep track of all the coefficients involved. \par
Moreover we can establish similar results for Ising models with external fields $h \ne 0$. Note that the major difference to the Ising model without external field is the loss of spin symmetry, i.e. the map $\sigma \mapsto -\sigma$ does not preserve the measure $q^n$ (more precisely, the push-forward is an Ising model with external field $-h$), and hence in general all homogeneous polynomials of odd degree are not centered random variables anymore. To overcome this obstruction we can recover concentration results for polynomial functions in $\tilde{X_i} \coloneqq X_i - \IE_{q^n} X_i$. To this end, define the (generalized) diagonal as 
\[
	\Delta_d \coloneqq \{ (i_1,\ldots, i_d) \in \{1,\ldots,n\}^d : \abs{\{i_1,\ldots,i_d\}} < d \}.
\]
and call a tensor $A = (A_{i_1,\ldots,i_d})_{i_1,\ldots,i_d = 1,\ldots, n}$ symmetric if $A_{i_1,\ldots, i_d} = A_{\pi(i_1),\ldots, \pi(i_d)}$ for any permutation $\pi \in S_d$. 
For notational convenience, let us write for any subset $I = \{i_1, \ldots, i_d \} \subset \{1,\ldots,n\}$ the product $X_I \coloneqq \prod_{i \in I} X_i$. We shall stick to the following four cases. Let $d \in \{1,\ldots, 4\}$ and define for any $d$-tensor $A = (a_{i_1,\ldots,i_d})_{i_1,\ldots,i_d=1,\ldots,n}$ with vanishing diagonal the functions
\begin{align*}
	f_{1,A}(X) &= \sum_{i = 1}^n a_i \tilde{X_i}, \\
	f_{2,A}(X) &= \sum_{i,j=1}^n a_{ij} (\tilde{X}_{ij} - \IE \tilde{X}_{ij}), \\
	f_{3,A}(X) &= \sum_{i,j,k = 1}^n a_{ijk} \left( \tilde{X}_{ijk} - \IE \tilde{X}_{ijk} - 3 \tilde{X}_i \IE( \tilde{X}_{jk} ) \right), \\
	f_{4,A}(X) &= \sum_{i,j,k,l = 1}^n a_{ijkl} \left( \tilde{X}_{ijkl} - \IE \tilde{X}_{ijkl} - 4 \tilde{X_i} \IE  \tilde{X}_{jkl} - 6 \tilde{X}_{ij} \IE \tilde{X}_{kl} + 6 \IE \tilde{X}_{ij} \IE \tilde{X}_{kl} \right).
\end{align*}

\begin{thm}	\label{Ising:Thm:ConcPolynomialsExternalField}
	Let $q^n$ be an Ising model as in Proposition \ref{Ising:Prop:LSI}, with an external field $h$. Let $d \in \{1,2,3,4\}$ be fixed, $A = (A_{i_1,\ldots,i_d})_{i_1,\ldots,i_d=1,\ldots,n}$ a symmetric tensor with vanishing diagonal and $f_{d,A}$ as above. For some constant $C = C(\alpha,\beta,d) > 0$ we have
	\begin{equation}
		q^n\left( \abs{f_{d,A} - \IE_{q^n} f_{d,A}} > t \right) \le 2 \exp \left( - \frac{t^{2/d}}{C \norm{A}_{HS}^{2/d}} \right) \le 2 \exp \left( - \frac{t^{2/d}}{Cn \norm{A}_\infty^{2/d}} \right).
	\end{equation}
\end{thm}

\begin{rem*}
	Note that Theorem \ref{Ising:Thm:ConcPolynomialsExternalField} can be extended to arbitrary $d \in \IN$, i.e. there exists a sequence of polynomials $(g_d)_{d \ge 1}, g_d: \IR^d \to \IR$ (with $g_1 = g_2 = 0$) of order $d-2$ with the property that
	\begin{align*}
		\mathfrak{h}_j (\sum_{\abs{I} = d} a_{I} (\tilde{X}_{I} - g_d(\tilde{X}_I))) = d \abs*{\sum_{\abs{I} = d-1} a^{(j)}_{I} ((\tilde{X}_{I} - \IE (\tilde{X}_{I})) - (g_{d-1}(\tilde{X}_I) - \IE g_{d-1}(\tilde{X}_I)))}.
	\end{align*}
	From this, the recursion in the proof can be extended to arbitrary $d \in \IN$ and the concentration result as well. But even for $d = 5$ this will be cumbersome to formulate, since one has to keep track of all the expectations involved to ensure that all the ``partial derivatives'' are centered for any degree up to $d-1$.
\end{rem*}

\subsection{General results}	\label{section:GeneralResult}
The results for the Ising model are an application of our main results. Let us $\abs{\partial f} = \left( \sum_{I \in \mathcal{I}} (\partial_i f)^2 \right)^{1/2}$ (associated to some set $\mathcal{I}$). For measures $\mu$ satisfying $\text{LSI}_{(\partial, \mathcal{I})}(\sigma^2)$ we derive moment inequalities which relate the $L^p(\mu)$-norms of functions $f$ with $L^p(\mu)$ norms of their differences $\abs{\partial f}$. This leads to a concentration of measure of higher order for functions with bounded differences of higher order. 

\begin{thm}	\label{Thm:d-th-order-conc}
	Let $d \in \mathbb{N}$, assume that $\mu$ satisfies $LSI_{(\partial,\mathcal{I})}(\sigma^2)$ with constant $\sigma^2 > 0$ and let $f \in L^\infty(\mu)$. Assuming the conditions
	\begin{equation}	\label{eqn:d-th-order-conc-cond1}
	\lVert \mathfrak{h}^{(k)} f \rVert_{2} \le \min(1, \sigma^{d-k})\quad \text{for all } k = 1, \ldots, d-1
	\end{equation}
	and
	\begin{equation}	\label{eqn:d-th-order-conc-cond2}
	\lVert\mathfrak{h}^{(d)} f\rVert_{\infty} \le 1,
	\end{equation}
	there exists some universal constant $c > 0$ such that
	\[
		\mathbb{E}_\mu \exp\left(c |f - \IE_\mu f|^{2/d}\right) \le 2.
	\]
	A possible choice is $c = 1/(12 \sigma^2 e)$.
\end{thm}

Since we are interested in the asymptotics for large $n$, the logarithmic Sobolev constant $\sigma^2$ might depend on $n$ and thus the constant $c$ in Theorem \ref{Thm:d-th-order-conc} might also depend on $n$. However, if the logarithmic Sobolev constant is independent of $n$, one may rewrite condition \eqref{eqn:d-th-order-conc-cond1} as
	\begin{equation}	\label{eqn:d-th-order-conc-cond1-alt}
	\lVert \mathfrak{h}^{(k)} f \rVert_{2} \le 1	\quad \text{for all } k = 1, \ldots, d-1
	\end{equation}	
Moreover, note that here one needs to control the first $d-1$ differences, but since we need bounds for $L^2(\mu)$ norms, various tools like variance decomposition or Poincar{\'e} inequality are available to achieve this.

\subsection{Outline}
In section \ref{section:HigherOrderDiff} we motivate and define the difference operators and prove the main result Theorem \ref{Thm:d-th-order-conc} by estimating the growth of moments under a logarithmic Sobolev inequality. Section \ref{section:Applications} contains examples of measures satisfying a logarithmic Sobolev inequality with respect to the Gibbs sampler type Dirichlet form. In section \ref{section:Applications:Ising} we prove Theorems \ref{Ising:Thm:d-th-order-conc} and \ref{Ising:Thm:ConcPolynomials} as well as Proposition \ref{Ising:Prop:LSI} and show by way of example that a third-order polynomial in the Ising model is concentrated around a first order polynomial, and prove Theorem \ref{Ising:Thm:ConcPolynomialsExternalField}. Sections \ref{section:Applications:randompermutations} and \ref{section:Applications:BLSSE} serve to demonstrate how to interpret the logarithmic Sobolev inequality with respect to difference operators corresponding to $(\partial, \mathcal{I})$ in the cases of random walks generated by switchings on either the symmetric group and the Bernoulli-Laplace and symmetric simple exclusion process, to indicate possible further applications. Finally, in section \ref{section:ApproximateTensorization} we give a proof of an approximate tensorization result given by K. Marton.

\section{Higher order difference operators for dependent arguments} \label{section:HigherOrderDiff}
To facilitate notations, we will write for any vector $(x_1, \ldots, x_n)$ and any subset $I \subset \{1,\ldots, n\}$, $x_I \coloneqq (x_i)_{i \in I}$ and $\overline{x_I} \coloneqq (x_i)_{i in I^c}$,  $\overline{x_i} \coloneqq \overline{x_{\{i\}}}$, and given any vector $x = (x_k)_{k \in I}$ and $y = (y_k)_{k \in I^c}$, $(x,y)$ for the vector with $(x,y)_k = \begin{cases} x_k & k \in I \\ y_k & k \in I^c \end{cases}$. Consistently, we shall use the notation $S_I = \otimes_{i \in I} S_i$ and $\overline{S_I} \coloneqq \otimes_{i \in I^c} S_i$ and denote by $\overline{\pi_I}: S \to \overline{S_I}, x \mapsto \overline{x_I}$ the (projection) map and  by $\overline{\mu_I} \coloneqq \mu_{\overline{\pi_I}}$ the push-forward measure. \par
In order to define the difference operators we recall the disintegration theorem in a special form for product spaces (although not endowed with product probability measures)  required in our context. For the existence we refer to \cite[Chapter III]{DM78} and for a modern formulation to \cite[Theorem 5.3.1]{AGS08}. 

\begin{proposition}[Disintegration theorem for product spaces]	\label{Prop:DisintegrationTheorem}
Let $S_1, \ldots, S_n$ be Polish spaces, $S \coloneqq \otimes_{i = 1}^n S_i$ endowed with the Borel $\sigma$-algebra and a Borel probability measure $\mu$. There exists a Markov kernel $(m_{\overline{x_I}})_{\overline{x_I} \in \overline{S_I}}$ such that 
\[ 
	\mu(A) = \int m_{\overline{x_I}}(A) d\mu_{\overline{\pi_I}} (\overline{x_I}) \quad \text{for} \quad A \in \mathcal{B}(S). 	
\]
Moreover, the Markov kernel can be seen as a family of probability measures on $S_I$ and for any $f \in L^1(\mu)$ we have
\[
	\int f d\mu = \int_{\overline{S_I}} \int_{S_I} f(\overline{x_I}, y_I) dm_{\overline{x_I}}(y_I) d \overline{\mu}_{I}(\overline{x_I}).
\]
\end{proposition}

This decomposition of a measure into a part which depends on the coordinates in some subset $I \subset \{1,\ldots, n\}$ and a conditional probability given the variables $X_I$ will serve as a starting point for the definition of our difference operators as follows.

\begin{definition}	\label{FirstOrderDiff}
Let $S_1, \ldots, S_n$ be Polish spaces and $\mu$ a measure on $S = \otimes_{i = 1}^n S_i$. For each subset $I \subset \{1,\ldots,n\}$ let $m_{\overline{x_I}}$ be the Markov kernel from Proposition \ref{Prop:DisintegrationTheorem}. Let $\mathcal{I} \subset \mathcal{P}(\{1, \ldots, n \})$ be a set of subsets. 
\begin{enumerate}[(i)]
\item For any $f \in L^2(\mu)$ and any $I \in \mathcal{I}$, let
\[
	\partial_I f(x) \coloneqq \left( \frac{1}{2} \int (f(x) - f(\overline{x_I}, y_I))^2 dm_{\overline{x_I}}(y_I) \right)^{1/2} 
\]
and introduce $\partial f = (\partial_I f)_{I \in \mathcal{I}}$.
\item For any $f \in L^\infty(\mu)$ and any $I \in \mathcal{I}$, define 
\[
	\fh_I f(x) \coloneqq \frac{1}{\sqrt{2}} \norm{f(\overline{x_I},y_I) - f(\overline{x_I},z_I)}_{L^\infty\left(m_{\overline{x_I}} \otimes m_{\overline{x_I}}(y_I, z_I)\right)},
\]
and $\fh f \coloneqq (\fh_I f)_{I \in \mathcal{I}}$.
\end{enumerate}
\end{definition}

For either $\partial$ or $\fh$ we can define a difference operator by setting $\Gamma(f) = \abs{\partial f}$ or $\Gamma(f) = \abs{\fh f}$ for the Euclidean norm $\abs{\cdot}$ and call it the associated operator to $(\partial, \mathcal{I})$ or $(\mathfrak{h}, \mathcal{I})$ respectively. It is clear that $\Gamma$ satisfies $\Gamma(af+b) = \abs{a} \Gamma(f)$. \par

As already mentioned in the introduction, on the basis of $\mathfrak{h}$, we define for any $d \in \mathbb{N}$ and any $I_1, \ldots, I_d \in \mathcal{I}$
\begin{align} 	
\begin{split}	
\mathfrak{h}_{I_1 \ldots I_d}f = \mathfrak{h}_{I_1}(\mathfrak{h}_{I_2\ldots I_d}f), 
\end{split}
\end{align}
and tensors of $d$-th order differences $\mathfrak{h}^{(d)}f(x)$ with coordinates $\mathfrak{h}_{I_1\ldots I_d}f(x)$, and analogously for $\partial$. \par

\begin{rem*}
The quantity $\int \abs{\partial f}^2 d\mu$ has the interpretation of a Dirichlet form. Indeed, defining the Markov kernel $m_x(dy) = \frac{1}{\abs{\mathcal{I}}} \sum_{I \in \mathcal{I}} m_{\overline{x_I}}(dy)$, it can be shown by expanding $\frac{1}{2} \iint (f(x) - f(y))^2 dm_x(y) d\mu(x)$ that
\[
	\frac{1}{2 \abs{\mathcal{I}}} \int \abs{\partial f}^2 d\mu = \frac{1}{2 \abs{\mathcal{I}}} \sum_{I \in \mathcal{I}} \iint (f(x) - f(y))^2 m_{\overline{x_I}}(dy) \mu(dx) =  \skal{f, -Lf}_\mu,
\]
where $Lf(x) = \int f(y) - f(x) dm_x(y)$ is the Laplacian. Hence there is an intimate connection to a Markov chain viewpoint, i.e. there is a natural dynamics for which $\int \abs{\partial f}^2 d\mu$ is its Dirichlet form. \par
The special case given by $\mathcal{I} = \mathcal{I}_1 \coloneqq \{ \{i \}, i = 1, \ldots, n \}$ translates into the disintegration with respect to $n-1$ variables and is well known, since the dynamics corresponds to the Glauber dynamics. Here, $\partial f$ and $\mathfrak{h}f$ are vectors in $\IR^n$. In probabilistic terms the definition of $\mathfrak{h}_i f(x)$ can be interpreted as an upper bound on the difference of $f$ if one updates the coordinate $i$, conditional on $\overline{x_i}$ being fixed. Moreover $\mathfrak{h}$ already appeared in the works of C. McDiarmid on concentration inequalities for functions with bounded differences, see e.g. \cite{McD89}. Here $\mathfrak{h}_i f(x)$ can still fluctuate and does not need to be bounded, resulting in possibly non-Gaussian concentration. \par
 In some cases, $\mathfrak{h}_I f$ is a function which depends on the coordinates $\overline{x_I}$ only, e.g. if all the measures $m_{\overline{x_I}}$ have full support. However, we would like to stress that in general the supports do not agree for different $\overline{x_I}$ and thus the supremum might depend on $x_I$, especially in situations which incorporate some kind of exclusion. A typical example is the disintegration of the measure on $\{1,\ldots, n\}^n$ given by the push-forward of the uniform random permutation under $\sigma \mapsto (\sigma(i))_{i \in \{1,\ldots,n\}}$, for which any disintegration is a Dirac measure on one point, see also section \ref{section:Applications:randompermutations}, and more generally for any $I \subset \{1,\ldots,n\}$ the Markov kernel $m_{\overline{x_I}}$ is concentrated on $\{1,\ldots,n\} \backslash \{ x_I \}$. \par
In the independent case, it is unnecessary to use the disintegration theorem for Polish spaces. Instead, one can simply define $m_{\overline{x_I}} = \otimes_{i \in I} \mu_i$ independent of $\overline{x_I}$, see the previous results by S.G. Bobkov, F. G\"otze and H. Sambale \cite{BGS17}. The definitions then coincide.
\end{rem*}

To prove Theorem \ref{Thm:d-th-order-conc} we shall need two ingredients: a pointwise estimate on consecutive differences as well as control on the growth of moments under a logarithmic Sobolev inequality.  
\begin{lemma}	\label{Lemma:pointwise-estimate-difference}
	For any $f \in L^\infty(\mu)$ and any $d \ge 1$ we have the pointwise estimate
	\begin{equation}
		\abs{\mathfrak{h} \abs{\mathfrak{h}^{(d)}f} } \le \abs{\mathfrak{h}^{(d+1)} f}.
	\end{equation}
\end{lemma}

\begin{proof}
	Let $I \in \mathcal{I}$ and $x \in S$ be fixed and write $\norm{\cdot}_{I,x}$ for $L^\infty(m_{\overline{x_I}} \otimes m_{\overline{x_I}})$. Using the reverse triangle inequality for $\abs{\cdot}$ and the triangle inequality for $\norm{\cdot}_{I,x}$ we obtain
	\begin{align*}
	(\fh_I \abs{\mathfrak{h}^{(d)} f})^2 &= \frac{1}{2} \norm*{\abs*{\fh^{(d)} f}(\overline{x_I},y) - \abs*{\fh^{(d)} f}(\overline{x_I},z)}_{I,x}^2 \\ 
	&\le \frac{1}{2} \norm*{\abs*{\fh^{(d)} f(\overline{x_I},y) - \fh^{(d)} f(\overline{x_I}, z)}^2}_{I,x} \\
	&= \frac{1}{2^{d+1}}\norm*{\sum_{I_1,\ldots, I_d } \left( \fh_{I_1\ldots I_d} f(\overline{x_I}, y) - \fh_{I_1\ldots I_d} f(\overline{x_I},z) \right)^2 }_{I,x} \\
	&\le \frac{1}{2^{d+1}} \sum_{\substack{I_1,\ldots,I_d}} \left( \mathfrak{h}_I \mathfrak{h}_{I_1\ldots I_d} f \right)^2.
	\end{align*}
	Summing over $I \in \mathcal{I}$ and taking the square root yields the result.
\end{proof}

By an adaption of the case of functions on finite graphs considered by S.\,G. Bobkov \cite[Theorem 2.1]{Bob10}, which in turn is based on arguments going back to L. Gross \cite{Gr75} as well as S. Aida and D. Stroock \cite{AS94}, we have the following result.

\begin{proposition}	\label{Prop:momentinequality}
	Let $\mu$ be a measure on a product space of Polish spaces satisfying $LSI_{(\partial, \mathcal{I})}(\sigma^2)$ with constant $\sigma^2 > 0$. Then, for any $f \in L^\infty(\mu)$ and any $p \ge 2$, we have
	\begin{equation}	\label{eqn:momentinequality2}
	\lVert f \rVert_p^2 - \lVert f \rVert_2^2 \le 2 \sigma^2 (p-2) \lVert \partial f \rVert_p^2
	\end{equation}
	as well as
	\begin{equation}	\label{eqn:momentinequality3}
	\lVert f \rVert_p^2 - \lVert f \rVert_2^2 \le 2 \sigma^2 (p-2) \lVert \mathfrak{h} f \rVert_p^2.
	\end{equation}
\end{proposition}

\begin{rem*}
	Actually, up to a constant, $LSI_{(\partial,\mathcal{I})}(\sigma^2)$ is equivalent to \eqref{eqn:momentinequality2}, which has also been remarked by S. G. Bobkov in \cite{Bob10}.
\end{rem*}

\begin{proof}
	Let $p > 0$, and let $f$ be any measurable function on an arbitrary probability space such that $0 < \lVert f \rVert_{p+\varepsilon} < \infty$ for some
	$\varepsilon > 0$. Then, we have the general formula
	\begin{equation}
		\frac{d}{dp} \lVert f \rVert_p = \frac{1}{p^2} \lVert f \rVert_p^{1-p} \Ent (|f|^p).
	\end{equation}
	In particular, it follows that
	\begin{equation}	\label{eqn:derivativeLpsquared}
		\frac{d}{dp} \lVert f \rVert_p^2 = \frac{2}{p^2} \lVert f \rVert_p^{2-p} \Ent (|f|^p).
	\end{equation}
	Moreover, note that for any $I \in \mathcal{I}$
	\begin{align*}
		\mathbb{E}_\mu (\partial_I f)^2 &= \frac{1}{2} \iint (f(x) - f(\overline{x_I}, y_I))^2 dm_{\overline{x_I}}(y_I) d\mu(x) = \int \Var_{m_{\overline{x_I}}}(f(\overline{x_I}, \cdot)) d\overline{\mu_I}(\overline{x_I}) \\
		&= \iiint (f(\overline{x_I}, y_I) - f(\overline{x_I},z_I))_+^2 dm_{\overline{x_I}}(z_I) dm_{\overline{x_I}}(y_I) d\overline{\mu_I}(\overline{x_I}) \\
		&= \iint (f(x) - f(\overline{x_I}, z_I))_+^2 dm_{\overline{x_I}}(z) d\mu(x).
	\end{align*}
	Therefore, it follows that
	\begin{equation}	\label{eqn:DarstellungPartialfSquared}
		\mathbb{E}_\mu |\partial f|^2 =  \sum_{I \in \mathcal{I}} \iint \left( f(x) - f(\overline{x_I},z_I) \right)_+^2 dm_{\overline{x_I}}(z_I) d\mu(x)
	\end{equation}
	Now let $p > 2$ and $f$ be non-constant. (The assumption $\lVert f \rVert_{p+\varepsilon} < \infty$ is always true since $f \in L^\infty(\mu)$.) Applying the logarithmic Sobolev inequality \eqref{LSI} to the function $|f|^{p/2}$ and rewriting this in terms of $\eqref{eqn:DarstellungPartialfSquared}$ yields
	\begin{align}
		\Ent (|f|^p) &\le 2\sigma^2 \sum_{I \in \mathcal{I}} \int \left( \int \left( \abs{f}^{p/2}(x) - \abs{f}^{p/2}(\overline{x_I},y_I) \right)_+^2 dm_{\overline{x_I}}(y_I) \right) d\mu(x) \label{eqn:EntFp01} \\
		&= 2\sigma^2 \sum_{I \in \mathcal{I}} \iiint (\abs{f}^{p/2}(\overline{x_I}, x_I) - \abs{f}^{p/2}(\overline{x_I}, y_I))_+^2 dm_{\overline{x_I}}(x_I) dm_{\overline{x_I}}(y_I) d\overline{\mu_I}(\overline{x_I}) \label{eqn:EntFp02}.
	\end{align}
	Using the inequality $(a^{p/2} - b^{p/2})_+^2 \le \frac{p^2}{4} a^{p-2}(a-b)^2$ for all $a,b \ge 0$ and all $p \ge 2$, we obtain
	$$(|f|^{p/2} - |f|^{p/2}(\overline{x_I},y_I))_+^2 \le \frac{p^2}{4} (|f| - |f|(\overline{x_I},y_I))_+^2 |f|^{p-2} \le \frac{p^2}{4} (f - f(\overline{x_I},y_I))^2 |f|^{p-2},$$
	from which it follows in combination with \eqref{eqn:EntFp01} that
	\begin{equation*}
	\Ent(\abs{f}^p) \le p^2 \sigma^2 \int \abs{f}^{p-2} \sum_{I \in \mathcal{I}} (\mathfrak{d}_I f)^2 d\mu = p^2 \sigma^2 \IE_\mu \abs{f}^{p-2} \abs{\mathfrak{d}f}^2
	\end{equation*}
	and in combination with \eqref{eqn:EntFp02} that
	\[
		\Ent (|f|^p) \le p^2 \sigma^2 \mathbb{E}_\mu \lvert f \rvert^{p-2} \lvert \mathfrak{h} f \lvert^2.
	\]
	H\"older's inequality with exponents $\frac{p}{2}$ and $\frac{p}{p-2}$ applied to the last integral yields
	$$\Ent (|f|^p) \le p^2 \sigma^2 \lVert \partial f \rVert_p^2 \lVert f \rVert_p^{p-2}$$
	or
	\[
		\Ent (|f|^p) \le p^2 \sigma^2 \lVert \mathfrak{h} f \rVert_p^2 \lVert f \rVert_p^{p-2}
	\]
	respectively. Combining this with \eqref{eqn:derivativeLpsquared}, we arrive at the differential inequality $\frac{d}{dp} \lVert f \rVert_p^2 \le 2 \sigma^2
	\lVert \partial f \rVert_p^2$ or $\frac{d}{dp} \lVert f \rVert_p^2 \le 2 \sigma^2 \lVert \mathfrak{h} f \rVert_p^2$ respectively, which after integration gives \eqref{eqn:momentinequality2} and \eqref{eqn:momentinequality3}.
\end{proof}

We shall prove Theorem \ref{Thm:d-th-order-conc} by estimating the growth of moments under the conditions in the following way. Recall that if a real-valued function $f$ on some probability space $(\Omega, \mathcal{A}, \mathbb{P})$ satisfies
\begin{equation}	\label{eqn:subexpAndMomentConditions}
	\norm{f}_k \le \gamma k
\end{equation}
for any $k \in \mathbb{N}$ and some constant $\gamma > 0$, it has sub-exponential tails, i.\,e. 
\begin{equation}	\label{eqn:ExpMoment}
	\mathbb{E} e^{c|f|} \le 2 \text{ for some constant } c = c(\gamma) > 0.
\end{equation}
Here, one may take $c = \frac{1}{2 \gamma e}$. Indeed, for any $c > 0$, we have
\[
	\mathbb{E} \exp(c|f|) = 1 + \sum_{k=1}^{\infty} c^k \frac{\mathbb{E} |f|^k}{k!} \le 1 + \sum_{k=1}^{\infty} (c \gamma)^k \frac{k^k}{k!} \le 1 + \sum_{k=1}^{\infty} (c \gamma e)^k,
\]
where the last inequality follows from the fact that $k! \ge (\frac{k}{e})^k$ for all $k \in \mathbb{N}$. Inserting $c = \frac{1}{2 \gamma e}$ we arrive at \eqref{eqn:ExpMoment}.

\begin{proof}[Proof of Theorem \ref{Thm:d-th-order-conc}]
	First let $p \ge 2$. Using \eqref{eqn:momentinequality3} with $f$ replaced by $\lvert \mathfrak{h}^{(k-1)} f \rvert$ for $k = 1, \ldots, d$ and Lemma \ref{Lemma:pointwise-estimate-difference} in the second step gives
	\begin{align*}
	\lVert \mathfrak{h}^{(k-1)}f \rVert_{p}^2 &\le \lVert \mathfrak{h}^{(k-1)}f \rVert_{2}^2 + 2\sigma^2(p-2) \lVert \mathfrak{h} \lvert \mathfrak{h}^{(k-1)}f \rvert \rVert_p^2\\
	&\le \lVert \mathfrak{h}^{(k-1)}f \rVert_{2}^2 + 2 \sigma^2 (p-2) \lVert \mathfrak{h}^{(k)}f \rVert_{p}^2.
	\end{align*}
	Consequently, by iteration and applying the Poincar{\'e} inequality for $\mathfrak{h}$ we arrive at
	\begin{align}
	\lVert f \rVert_p^2 &\le \lVert f \rVert_2^2 + \sum_{k=1}^{d-1} (2\sigma^2(p-2))^k \lVert \mathfrak{h}^{(k)}f \rVert_{2}^2
	+ (2 \sigma^2 (p-2))^d \lVert \mathfrak{h}^{(d)}f \rVert_{p}^2\notag\\
	&\le \sigma^2 \lVert \mathfrak{h}f \rVert_2^2 + \sum_{k=1}^{d-1} (2 \sigma^2 (p-2))^k \lVert \mathfrak{h}^{(k)}f \rVert_{2}^2
	+ (2 \sigma^2 (p-2))^d \lVert \mathfrak{h}^{(d)}f \rVert_{p}^2\notag\\
	&\le \sum_{k=1}^{d-1} (2 \sigma^2 p)^k \lVert \mathfrak{h}^{(k)}f \rVert_{2}^2 + (2 \sigma^2 p)^d \lVert \mathfrak{h}^{(d)}f \rVert_{p}^2. \notag
	\end{align}
	Now, since $\lVert \mathfrak{h}^{(k)}f \rVert_{2} \le \min(1, \sigma^{d-k})$ for all $k = 1, \ldots, d-1$ and $\lVert \mathfrak{h}^{(d)}f \rVert_{\infty} \le 1$ by assumption, we obtain
	$$\lVert f \rVert_p^2 \le \sigma^{2d} \sum_{k=1}^{d} (2p)^k \le \frac{1}{1 - (2p)^{-1}} (2 \sigma^2 p)^d \le (3 \sigma^2 p)^d$$
	and therefore
	$$\lVert f \rVert_p \le (3 \sigma^2 p)^{d/2}$$
	Moreover, for all $p < 2$, by H\"older's and Jensen's inequality we have
	$$\lVert f \rVert_p \le \lVert f \rVert_2 \le (6 \sigma^2)^{d/2}.$$
	Considering $p = 2k/d$, $k = 1, 2, \ldots$ yields
	$$\lVert \lvert f \rvert^{2/d} \rVert_k \le 6 \sigma^2 \frac{1}{d} k,\qquad k \ge d,$$
	and
	$$\lVert \lvert f \rvert^{2/d} \rVert_k \le 6 \sigma^2 = 6 \sigma^2 \frac{1}{k}k,\qquad k \le d-1.$$
	It follows that
	$$\lVert \lvert f \rvert^{2/d} \rVert_k \le \gamma k$$
	for all $k \in \mathbb{N}$, where $\gamma = 6 \sigma^2 \max(1, 1/2, \ldots, 1/(d-1), 1/d) = 6 \sigma^2$.
	In view of \eqref{eqn:subexpAndMomentConditions}, this completes the proof.
\end{proof}

\section{Applications}		\label{section:Applications}
\subsection{Ising model}	\label{section:Applications:Ising}
Let $S_n = \{ -1,+1 \}^n$ be the configuration space of the Ising model on $n$ sites, $J = (J_{ij})$ a symmetric matrix with vanishing diagonal, $h \in \IR^n$ and define $\pi: S_n \to \IR$ via
\begin{equation} 
    \pi(\sigma) = \exp\left(\frac{1}{2}\skal{\sigma, J\sigma} + \skal{h,\sigma} \right) = \exp\left( \frac{1}{2} \sum_{i,j} J_{ij} \sigma_i \sigma_j + \sum_{i=1}^n h_i \sigma_i \right). 
\end{equation}
Equip $S_n$ with the Gibbs measure $q^n(\sigma) = Z^{-1} \pi(\sigma)$, with $Z$ being the normalization constant. For each $i \in \{1, \ldots, n\}$ denote by $T_i: S_n \to S_n$ the operator which switches the sign of the $i$-th coordinate.

\begin{rem*}
	The factor $\frac{1}{2}$ corresponds to the fact that we made the matrix symmetric, i.e. $J = \tilde{J} + \tilde{J}^T$, where $\tilde{J}$ is the upper triangular matrix. This is consistent with the Curie-Weiss model in \cite[Example 2.1]{CD10} or \cite{BN17}, but not with \cite{GLP17}.
\end{rem*}

We would like to use an approximate tensorization of entropy result proven by K. Marton in \cite{Ma15} and the results from the last section to obtain concentration inequalities for polynomials in weakly dependent random variables, more specifically for Ising models which are sufficiently close to being product measures, i.e. which satisfy the condition of Proposition \ref{Ising:Prop:LSI}. The local specifications of the Ising model (i.e. the conditional probabilities) $q_i(\cdot \mid \overline{\sigma_i})$ for $\overline{\sigma_i} \in \overline{S_i}$ are given by 
\begin{equation} \label{Ising:eqn:conditionalProb}
	q_i(\cdot \mid \overline{\sigma_i}) = \frac{\pi(\overline{\sigma_i}, \cdot)}{\pi(\overline{\sigma_i}, 1) + \pi(\overline{\sigma_i}, -1)},
\end{equation}
which can be written as
\begin{equation}	\label{Ising:eqn:condtionalProbWithtanh}
	q_i(1 \mid \overline{\sigma_i}) = \frac{1}{2} \left( 1 + \tanh(\sigma_i \sum_j J_{ij} \sigma_j + h_i \sigma_i) \right).
\end{equation}

More generally, given any $I \subset \{1,\ldots,n\}$, we can define $q_I(\cdot \mid \overline{\sigma_I})$ as the probability measure on $\{ -1, +1\}^I$ given by normalizing $q^n(\cdot, \overline{\sigma_I})$. For $I = \{1,\ldots,n\} \backslash \{ j \}$ we also write $\overline{q_j}$. \par
In \cite{Ma15}, the author proves an approximate tensorization property of the relative entropy with respect to a fixed measure $q^n$ (which in our case will be the Gibbs measure given above) in the sense that
\begin{align}
	\Ent_{q^n}(f) \le \frac{2C}{\beta} \sum_{i = 1}^n \int \Ent_{q_i(\cdot \mid \overline{y_i})}(f(\overline{y_i}, \cdot)) d\overline{q_i}(\overline{y_i})	
\end{align}
holds under certain conditions. Here $\beta$ is the minimal conditional probability and $C$ is a constant which depends on the interdependence matrix. However in the proof of \cite[Theorem 1]{Ma15} there is a small oversight, hence (and for the sake of completeness) we include a full exposition of the proof in section \ref{section:ApproximateTensorization}, see Theorem \ref{AT:Thm:ApproximateTensorization}. Moreover, \cite[Theorem 2]{Ma15} replaces one of the conditions of \cite[Theorem 1]{Ma15} by another condition, which is easier to check, see Theorem \ref{AT:Thm:ApproximateTensorization} $(iii)$. Indeed, this condition holds via bounds on the operator norm of the coupling matrix $A = (A_{ik})_{i \neq k}$ defined as
	\[
	A_{ik} \coloneqq \sup_{\substack{x,z\in S_n \\ \overline{x_k} = \overline{z_k}}} d_{TV}\left( q_i(\cdot \mid \overline{x_i}), q_i(\cdot \mid \overline{z_i}) \right).
	\] 
Thus, provided that $\norm{A}_{2 \to 2} < 1$, an approximate tensorization property holds with $C = (1-\norm{A}_{2 \to 2})^{-2}$.

\begin{lemma}	\label{Ising:Lemma:Prop15impliesMarton}
	Let $q^n$ be an Ising model with an interaction matrix $J$ satisfying $J_{ii} = 0$ and $\norm{J}_{1 \to 1} \le 1-\alpha$. Then $\norm{A}_{2 \to 2} \le \norm{J}_{1 \to 1} \le 1-\alpha$ holds for $q^n$. \par
	Moreover, if $\abs{h} \le \tilde{\alpha}$ for some $\tilde{\alpha}$ independent of $n$, then
	\[ q_i(\cdot \mid \overline{\sigma_i}) \in (c_{\alpha,\tilde{\alpha}}, 1-C_{\alpha,\tilde{\alpha}}) \] 
	for some $c_{\alpha,\tilde{\alpha}}, C_{\alpha,\tilde{\alpha}}$ depending only on $\alpha$ and $\tilde{\alpha}$, uniformly in $i, n$ and $\overline{\sigma_i}$.
\end{lemma}

\begin{proof}
Let $i \neq k$ be fixed and $z,y \in S_n$ be such that $y$ and $z$ differ in the $k$-th coordinate only, i.e. $y = T_k z$. Define $\sigma \coloneqq (\overline{z_i}, 1)$ and $h_i(\sigma) \coloneqq \sigma_i \sum_j J_{ij} \sigma_j + h_i \sigma_i$. We have by equation \eqref{Ising:eqn:condtionalProbWithtanh} and the $1$-Lipschitz property of $\tanh$
\begin{align*}
	d_{TV}\left( q_i( \cdot \mid \overline{z_i}), q_i( \cdot \mid \overline{y_i}) \right) &= \abs{q_i(1 \mid \overline{z_i}) - q_i(1 \mid \overline{y_i})} 
    = \frac{1}{2} \abs{\tanh(h_i(\sigma)) - \tanh(h_i(T_k \sigma))} \\
    &\le \frac{1}{2} \abs{h_i(\sigma) - h_i(T_k \sigma)} = \abs{J_{ki}}. 
\end{align*}
Thus we have $A_{ij} \le \abs{J_{ij}}$. \par 
The inequality is a simple consequence of 
\[
    \norm{A}_{2 \to 2} \le \sqrt{\norm{A}_{1 \to 1} \norm{A^T}_{1 \to 1}} \le  \sqrt{\norm{J}_{1 \to 1}\norm{J^T}_{1 \to 1}} \le 1-\alpha
\]
which follows from the general estimate $\abs{\lambda_i(JJ^T)} \le \norm{JJ^T} \le \norm{J} \norm{J^T}$ for any operator norm and $J = J^T$. \par
The second statement follows easily by using equation \eqref{Ising:eqn:condtionalProbWithtanh} and the estimates $\max_i \abs{h_i(\sigma)} \le \norm{J}_{1 \to 1} \le 1-\alpha$ and $\abs{h_i} \le \tilde{\alpha}$.
\end{proof}

To be able to prove Proposition \ref{Ising:Prop:LSI}, we will require analogue of Proposition \ref{Prop:momentinequality} for Markov kernels. This will be used on the ``local level'' after the tensorization procedure, enabling us to derive both inequalities \eqref{Ising:eqn:LSI} and \eqref{Ising:eqn:momentinequality}.

\begin{lemma}	\label{Ising:Lemma:MarkovKernelLSI}
	Let $K$ be a Markov kernel on a finite set $\mathcal{X}$, reversible w.r.t. $\pi$ and assume that $\pi$ satisfies a logarithmic Sobolev inequality with a constant $\sigma^2$, i.e.
	\[
		\Ent_\pi(f^2) \le 2\sigma^2 \int \sum_{y \in \mathcal{X}} (f(x) - f(y))^2 K(x,y) d\pi(x).
	\]
	For $p \ge 2$ we obtain
	\begin{align}
		\Ent_\pi(\abs{f}^p) &\le \sigma^2 p^2 \int \abs{f}^{p-2} \abs{\partial f}^2 d\pi \\
		\Ent_{\pi}(\abs{f}^p) &\le \sigma^2 p^2 \norm{f}_p^{p-2} \norm{\partial f}_p^2,
	\end{align}
	where $\abs{\partial f}(x) = \left( \sum_{y \in \mathcal{X}} (f(x) - f(y))^2 K(x,y) \right)^{1/2}$.
\end{lemma}

\begin{proof}
	Using that $(K,\mu)$ satisfies a logarithmic Sobolev inequality with constant $\sigma^2$, we obtain
	\begin{align*}
		\Ent_\pi(\abs{f}^p) &\le 2 \sigma^2 \int \sum_{y \in \mathcal{X}} \left( \abs{f}^{p/2}(x) - \abs{f}^{p/2}(y) \right)^2 K(x,y) d\pi(x) \\
		&\le 4 \sigma^2  \int \sum_{y \in \mathcal{X}} \left( \abs{f}^{p/2}(x) - \abs{f}^{p/2}(y) \right)^2_+ K(x,y) d\pi(x)
	\end{align*}
	where we have used $(a-b)^2 = (a-b)_+^2 + (b-a)_+^2$ and the fact that on $\{ f(y) > f(x) \}$, reversibility gives $\pi(x) K(x,y) = \pi(y) K(y,x)$ and we can exchange the roles of $x$ and $y$. For $a, b \ge 0$ the inequality $(a^{p/2} - b^{p/2})_+^2 \le \frac{p^2}{4} a^{p-2} (a-b)^2$ gives
	\[
		\Ent_\pi(\abs{f}^p) \le p^2 \sigma^2 \int \abs{f}^{p-2}(x) \sum_{y \in \mathcal{X}} (f(x) - f(y))^2 K(x,y) d\pi(x) = p^2 \sigma^2 \int \abs{f}^{p-2} \abs{\partial f}^2 d\pi.
	\]
	An application of H\"older's inequality yields the second inequality.
\end{proof}

We are now ready to prove Proposition \ref{Ising:Prop:LSI}, i.e. the logarithmic Sobolev inequality \eqref{Ising:eqn:LSI} and the moment inequality \eqref{Ising:eqn:momentinequality}.

\begin{proof}[Proof of Proposition \ref{Ising:Prop:LSI}]
	We can apply Lemma \ref{Ising:Lemma:Prop15impliesMarton} to see that by Theorem \ref{AT:Thm:ApproximateTensorization}(iii) we have for some $\beta = \beta(\alpha, \tilde{\alpha})$
	\begin{equation}	\label{Ising:eqn:ApproximateTensorization}
		\Ent_{q^n}(f^2) \le \frac{2}{\alpha^2 \beta} \sum_{i = 1}^n \int \Ent_{q_i(\cdot \mid \overline{y_i})}(f^2(\overline{y_i}, \cdot)) d\overline{q_i}(\overline{y_i}),
	\end{equation}
	so that it remains to find a uniform bound for the entropy given $\overline{y_i}$. To this end, fix $i \in \{1, \ldots, n \}$, $\overline{y_i} \in \{ -1, +1\}^{n-1}$ and to lighten notation write $q(\cdot) \coloneqq q_i(\cdot \mid \overline{y_i})$. $q$ is a measure on $\{ -1, +1\}$ and the Markov chain given by $K(x_0,x_1) = q(x_1)$ is reversible w.r.t. $q$. By \cite[Theorem A.1]{DSC96} (see also \cite[Example 3.8]{BT06}) $(K,q)$ satisfies a logarithmic Sobolev inequality with a constant depending on $q_* = \min_{x \in \{ -1, +1\}} q(x)$. However, this constant is bounded from below by Lemma \ref{Ising:Lemma:Prop15impliesMarton} uniformly in $\overline{y_i} \in S_{n-1}$ and $n \in \IN$. Thus, we have
	 \begin{align}	\label{Ising:1dimLSI}
	 	\Ent_q(f^2) \le C \iint (f(x) - f(y))^2 dq(x) dq(y). 
	 \end{align}
	 Inserting (\ref{Ising:1dimLSI}) into (\ref{Ising:eqn:ApproximateTensorization}) yields for some constant $C = C(\alpha, \tilde{\alpha})$
	 \begin{align*}
	 	\Ent_{q^n}(f^2) \le C \sum_{i = 1}^n \iiint (f(\overline{y_i},x) - f(\overline{y_i},y))^2 dq(x) dq(y) d\overline{q_i}(\overline{y_i}) = 2C \IE_{q^n} \abs{\partial f}^2,
	 \end{align*}
	 which proves a logarithmic Sobolev inequality for $q^n$. \par
	 To prove equation \eqref{Ising:eqn:momentinequality}, we shall make use of Lemma \ref{Ising:Lemma:MarkovKernelLSI} to first establish
	\begin{equation}	\label{Ising:eqn:EntVsPartial}
		\Ent_{q^n}(\abs{f}^p) \le C \frac{p^2}{2} \norm{f}_p^{p-2} \norm{{\partial} f}_p^2 \le C \frac{p^2}{2} \norm{f}_p^{p-2} \norm{\mathfrak{h} f}_p^2.
	\end{equation}	 
	Apply equation \eqref{Ising:eqn:ApproximateTensorization} to $\abs{f}^{p/2}$ to get
	\[
		\Ent_{q^n}(\abs{f}^p) \le C(\alpha) \sum_{i = 1}^n \IE_{\overline{q_i}} \Ent_{q_i(\cdot \mid \overline{y_i})}(\abs{f}^{p}(\overline{y_i}, \cdot)).
	\]
	Again by \cite[Theorem A.1]{DSC96} we obtain that the entropy with respect to the conditional measure has a uniformly bounded logarithmic Sobolev constant $\sigma^2(\alpha, \tilde{\alpha})$, and hence by Lemma \ref{Ising:Lemma:MarkovKernelLSI}
	\[
		\Ent_{q(\cdot \mid \overline{y_i})}(\abs{f}^{p}) \le 2\sigma^2 p^2 \int \abs{f}^{p-2}(x, \overline{y_i}) \abs{\partial f(x, \overline{y_i})}^2 dq_i(x \mid \overline{y_i}).
	\]
	Thus we can write
	\begin{align*}
		\Ent_{q^n}(\abs{f}^p) &\le C \frac{p^2}{2} \sum_{i = 1}^n \iint  \abs{\partial f(\overline{y_i}, \cdot)}^2 \abs{f(\overline{y_i}, \cdot)}^{p-2} dq_i(x \mid \overline{y_i}) d\overline{q_i}(\overline{y_i}) \\
		&= C \frac{p^2}{2} \int \abs{f(y)}^{p-2} \left( \sum_{i = 1}^n \int (f(y) - f(y_i, \overline{y_i}))^2 dq_i(y \mid \overline{y_i}) \right) dq^n(y) \\
		&= Cp^2 \IE_{q^n} \abs{f}^{p-2} \abs{\partial f}^2,
	\end{align*}
	and an application of H\"olders inequality yields equation \eqref{Ising:eqn:EntVsPartial}. \par
	Lastly, the proof of \eqref{Ising:eqn:momentinequality} is an easy adaption of the proof of Proposition \ref{Prop:momentinequality}, since the main argument was the inequality \eqref{Ising:eqn:EntVsPartial}.
\end{proof}

\begin{proof}[Proof of Theorem \ref{Ising:Thm:d-th-order-conc}]
	Theorem \ref{Ising:Thm:d-th-order-conc} is an application of Theorem \ref{Thm:d-th-order-conc}, since $q^n$ satisfies a logarithmic Sobolev inequality with respect to $\mathcal{I} = \{1,\ldots,n\}$.
\end{proof}

	
One can calculate using the reverse triangle inequality and the monotonicity of the square function as in the proof of Lemma \ref{Lemma:pointwise-estimate-difference} that for any $i_1 \neq i_2 \neq \ldots \neq i_d$
\begin{equation}	\label{Ising:eqn:mathfrakhi1id}
	(\mathfrak{h}_{i_1\ldots i_d} f)^2 \le \frac{1}{2^{d}} \abs*{\prod_{j = 1}^d (Id-T_{i_j}) f}^2
\end{equation}
holds, which also implies
\begin{equation}	\label{Ising:eqn:mathfrakhinequality}
	\abs{\mathfrak{h}^{(d)} f} \le \left( 2^{-d} \sum_{\abs{I} = d} \left(\left(\prod_{i \in I} (Id-T_i)\right)f\right)^2 \right)^{1/2},
\end{equation}
where with slight abuse of notation we write for any function $f: S_n \to \IR$ $T_i f$ for the function defined via $T_i f = f \circ T_i$ and where $T_{i_1 \ldots i_d}$ is defined via iteration. Note that on the right-hand side we deliberately chose summing over $\abs{I} = d$ instead of $i_1,\ldots,i_d$, since $\mathfrak{h}_{i_1\ldots i_d} f = 0$ if $i_j = i_k$ for some $j \neq k$. \par
For the operator appearing on the right-hand side of equation \eqref{Ising:eqn:mathfrakhinequality}, it was already shown by H. Sambale \cite{S16} and S. G. Bobkov, F. G\"otze and H. Sambale \cite[Lemma 2.2]{BGS17} that  the chain of pointwise inequalities from Lemma \ref{Lemma:pointwise-estimate-difference} holds.
%

Using this, one can infer the asymptotic behavior of $d$-th order polynomials in the spin variables of the Ising model with no external field. 

\begin{proof}[Proof of Theorem \ref{Ising:Thm:ConcPolynomials}]
	Let $f = \sum_{\abs{I} = d} a_I \sigma_I = \sum_{\abs{I} = d} a_I \prod_{i \in I} \sigma_i$ be a $d$-th order homogeneous polynomial and without loss of generality assume $\norm{A}_\infty = 1$. Consider the equation \eqref{Ising:eqn:momentinequality} from Proposition \ref{Ising:Prop:LSI}. A straightforward iteration in combination with the pointwise inequality between the $d$-th order differences from Lemma \ref{Lemma:pointwise-estimate-difference} yields
\begin{equation}	\label{Ising:eqn:f-IEflesumoverderivatives}
	\norm{f-\IE_{q^n} f}_p^2 \le \sum_{k = 1}^{d-1} p^k (2C(\alpha))^k \norm{\mathfrak{h}^{(k)} f}_2^2 + p^d (2C(\alpha))^d \norm{\mathfrak{h}^{(d)} f}_p^2.
\end{equation}
Now for any $k \in \{1,\ldots, d-1\}$ by equation \eqref{Ising:eqn:mathfrakhi1id} we have
\[
	(\mathfrak{h}_{i_1, \ldots, i_k} f)^2 \le 2^k \Big(\sum_{\substack{\abs{I} = d-k\\i_1, \ldots, i_k \notin I}} a_{I \cup i_1, \ldots, i_k} \sigma_I \Big)^2,
\]
and from \cite[Lemma 3.1]{GLP17} it follows that $\norm{\mathfrak{h}^{(k)} f}_2^2 = \sum_{i_1, \ldots, i_k} \norm{\mathfrak{h}_{i_1, \ldots, i_k} f}_2^2 \le c_k n^d$, since for each fixed $i_1,\ldots,i_k$ the integrand is a polynomial of degree at most $2(d-k)$ with coefficients bounded by $1$. Hence ultimately we obtain for any $p \ge 2$
\[
	\norm{f- \IE_{q^n} f}_p^2 \le n^d (2C(\alpha))^d \max(1, c_1, \ldots, c_{d-1}) \sum_{k = 1}^d p^k,
\]
which can be rewritten as
\[
	\norm{n^{-d/2}(f-\IE_{q^n} f)}_p \le C(\alpha,d) p^{d/2}
\]
with $C(\alpha,d) = (2C(\alpha))^{d/2}\max(1, c_1, \ldots, c_{d-1})^{1/2} d^{1/2}$, which is equivalent to the exponential integrability of $\abs{n^{-d/2}(f - \IE_{q^n} f)}^{2/d}$, i.e. for some constant $c > 0$ we have
\[
	\IE_{q^n} \exp \left( c \abs{n^{-d/2} (f- \IE_{q^n} f)}^{2/d} \right) \le 2,
\]
which by using Chebyshev's inequality results in 
\[
	q^n\left(n^{-d/2}\abs{f - \IE_{q^n} f} > t\right) \le 2 \exp\left( - \frac{t^{2/d}}{ \tilde{C}(\alpha)} \right)	
\]
for all $t > 0$, which is equivalent to the claim.
\end{proof}

\begin{rem*}
	Actually the equation \eqref{Ising:eqn:f-IEflesumoverderivatives} admits a more accurate estimate of the tail properties of $f - \IE_{q^n} f$, which has already been used in \cite[Theorem 7]{Ad06} and \cite[Theorem 3.3]{AW15}. It is based on the idea that by Chebyshev's inequality for any $p \ge 1$ we obtain
	\begin{equation}	\label{Ising:eqn:ChebyshevWithLp}
		q^n( \abs{f-\IE_{q^n} f} \ge e \norm{f-\IE_{q^n}f}_p) \le \exp(-p).
	\end{equation}
	First, observe that by taking the square root and using its subadditivity property in equation \eqref{Ising:eqn:f-IEflesumoverderivatives} we obtain 
	\[
		e\norm{f - \IE_{q^n} f}_p \le e \left( \sum_{k = 1}^{d-1} (2C(\alpha)p \norm{\mathfrak{h}^{(k)}f}_2^{2/k})^{k/2} + (2C(\alpha)p \norm{A}_2^{2/d})^{d/2} \right). 
	\]
	Now consider the function
	\[
		\eta_f(t) \coloneqq \min \left( \frac{t^{2/d}}{2C(\alpha)\norm{A}_2^{2/d}}, \min_{k = 1,\ldots,d-1} \frac{t^{2/k}}{2C(\alpha) \norm{\mathfrak{h}^{(k)} f}^{2/k}_2} \right)
	\]
	and assume that $\eta_f(t) \ge 2$, so that we can estimate
	\[
		\norm{f - \IE_{q^n}f}_{\eta_f(t)} \le \sum_{k = 1}^{d-1} t + t = (de)t.
	\]
	Applying equation \eqref{Ising:eqn:ChebyshevWithLp} to $p = \eta_f(t)$ (if $p \ge 2$)
	\[
		q^n(\abs{f- \IE_{q^n} f} \ge (de)t) \le q^n(\abs{f- \IE_{q^n} f} \ge d^{-1} \norm{f - \IE_{q^n}f}_{\eta_f(t)}) \le \exp\left( - \eta_f(t) \right)
	\]
	and combining it with the obvious estimate (in the case $p \le 2$) gives
	\[
		q^n(\abs{f- \IE_{q^n} f} \ge (de)t) \le e^2 \exp(-\eta_f(t)).
	\]
	To remove the $de$ factor, it is easiest to rescale the function by $\frac{1}{de}$ and estimating $\eta_{\frac{f}{de}}(t) \ge \frac{\eta_{f}(t)}{(de)^2}$.
\end{rem*}

Finally, let us give two examples on how to use the previous results in order to obtain more precise results on the concentration of a $d$-th order polynomial by approximating it with a lower-order polynomial. 

\begin{exa*}
	Let $A = (a_{ij})_{i,j}$ be a strictly upper triangular matrix and consider the function $\tilde{f}(\sigma) = \skal{\sigma, A\sigma} = \sum_{i<j} a_{ij} \sigma_i \sigma_j$ and $f = \tilde{f} - \IE_{q^n} \tilde{f}$.
	Defining $\tilde{a}_{ij} = a_{\min(i,j), \max(i,j)}$ we have
	\[
		\norm{\mathfrak{h} f}_2^2 \le 2 \sum_{i = 1}^n \int \left( \sum_{j = 1}^n \tilde{a}_{ij} \sigma_j \right)^2 dq^n = 2\sum_{i =1}^n \Var_{q^n} g_i \le 2C \sum_{i,j} \tilde{a}_{ij}^2 = 2C \norm{A}^2
	\]
	and 
	\[
	\mathfrak{h}_{kl} f(\sigma) \le \frac{1}{2} \abs{f(\sigma) - f(T_k \sigma) - f(T_l \sigma) + f(T_k T_l \sigma)} = 2 \abs{a_{kl}}.
	\]
	Thus we have
	\begin{align*}
		\norm{\mathfrak{h}^{(1)} f}_2 &\le 2C \norm{A} \\
		\norm{\mathfrak{h}^{(2)} f}_\infty &\le 2C \norm{A},
	\end{align*}
	so that after a renormalization by $\frac{1}{2C\norm{A}}$ the assumptions of Theorem \ref{Ising:Thm:d-th-order-conc} are satisfied and for any $t > 0$ we have
	\[
		q^n(\abs{f - \IE_{q^n} f} > t) \le 2 \exp \left( - \frac{t}{2 C\norm{A}} \right).
	\]
\end{exa*}

\begin{exa*}
	Similarly, with some modifications, one can show fluctuations of a third-order polynomial around a first-order polynomial in the following way. For any $3$-tensor $A = (a_{ijk})_{ijk}$ with the property that $a_{ijk} = 0$ unless $i < j < k$, define the matrix $\tilde{a}_{ij}^{(l)} = \begin{cases} a_{\pi(i),\pi(j), \pi(l)} & i < j \\ 0 & \text{otherwise} \end{cases}$, where $\pi$ is the unique permutation such that the three indices are ordered, and the vector $\tilde{a}_{k}^{(l,m)}$ similarly. Now let $\tilde{f} = \sum_{i,j,k} a_{ijk} \sigma_i \sigma_j \sigma_k$ and $f = \tilde{f} - \sum_{l = 1}^n \IE\left( \sum_{i,j} a_{ijl} \sigma_i \sigma_j \right) \sigma_l$. We claim that $\frac{1}{8 \norm{A}C} f$ satisfies the assumptions of Theorem \ref{Ising:Thm:d-th-order-conc} for $d = 3$. To this end, let us calculate the differences of all orders. First, using the Poincar{\'e} inequality gives
	\begin{align*}
		\norm{\mathfrak{h} f}_2^2 &= \sum_{l = 1}^n \IE_{q^n} \left( \mathfrak{h}_l f \right)^2 = 2 \sum_{l=1}^n \IE_{q^n} \left(\sum_{i,j} \tilde{a}^{(l)}_{ij} \left( \sigma_i \sigma_j -  \IE_{q^n} \sigma_i \sigma_j \right) \right)^2 \\
		&\le 2 C \sum_{l,k} \IE_{q^n} \left( g_l - g_l \circ T_k \right)^2 \le 4 C \sum_{l, k} \IE_{q^n} \left( \sum_i \tilde{a}^{(l,k)}_{i} \sigma_i \right)^2 \\
		&\le 8 C^2 \sum_{i,j,l} (\tilde{a}^{(j,l)}_{i})^2 = 8 C^2 \norm{A}^2
	\end{align*}
	as well as
	\begin{align*}
		\norm{\mathfrak{h}^{(2)} f}_2^2 = 4 \sum_{i,j} \int (\sum_k \tilde{a}^{(i,j)}_{k} \sigma_k)^2 dq^n(\sigma) \le 8C \sum_{i,j,k} (\tilde{a}^{(i,j)}_{k})^2 = 8 C \norm{A}^2.
	\end{align*}
  	Additionally, we have $\abs{\mathfrak{h}^{(3)} f}^2(\sigma) = 8 \sum_{i,k,l} a_{ikl}^2 = 8 \norm{A}^2$. Thus a normalization given by $\frac{1}{8C\norm{A}}$ is sufficient to apply Theorem \ref{Ising:Thm:d-th-order-conc}, which implies
  	\[
  		q^n\left( \abs{f} > t \right) \le 2 \exp \left( - \frac{t^{2/3}}{C \norm{A}^{2/3}} \right).
  	\]
\end{exa*}

\begin{rem*}
	The second example has an interesting interpretation since it shows that a polynomial of order three is not concentrated around its mean (which in this case would be zero), but around a first-order correction. For the case 
	\[ a_{ijk} = \begin{cases} \frac{1}{n^{3/2}} &i \neq j \neq k \\ 0 & \text{otherwise} \end{cases} \] 
	we obtain $f = n^{-1/2} \sum_{i = 1}^n \sigma_i \left( n^{-1} \sum_{j \neq i, k \neq i} (\sigma_j \sigma_k - \IE_{q^n} \sigma_j \sigma_k) \right) = n^{-1/2} \sum_{i = 1}^n \sigma_i c_i(\overline{\sigma_i})$. For $c_i$ independent of $\sigma$, first-order results of K. Marton \cite{Ma03} or the method of exchangeable pairs by S. Chatterjee \cite{Ch07} would imply that $f$ is subgaussian with variance $\norm{c}^2$. In this case, the variance fluctuates as well, and has exponential tails, and the normalization $n^{-1}$ ensures that this is the correct scaling order.
\end{rem*}

The concentration result of the second example leads to a special case of Theorem \ref{Ising:Thm:ConcPolynomials}, since the first-order correction can be controlled, as it concentrates on a different scale. However, since the coefficients in the first-order correction are growing, one needs to restrict the range for which one can expect to have stretched-exponential tails. By way of example, for $d = 3$ we obtain the easy corollary.

\begin{cor}
	There exist constants $C_1, C_2$ depending on $\alpha$ such that for all third order polynomials $f = \sum_{i,j,k} a_{ijk} \sigma_i \sigma_j \sigma_k$ with $\norm{A}_\infty \le 1$ and $a_{ijk} = 0$ if $\abs{\{ i,j,k \}} \neq 3$ and for any $t > 2C_1 n^{3/2}$ we have
	\[
		q^n(\abs{f} > t) \le 4 \exp\left( -\frac{t^{2/3}}{2C_2 n} \right).
	\]
\end{cor}

\begin{proof}
	Write $f_1 \coloneqq \sum_i \sigma_i c_i$, where $c_i = \sum_{j,k} a_{ijk} \IE_{q^n} \sigma_j \sigma_k$, for the first-order correction to  $f$. Observe that we have $\norm{A}^{2/3} \le n$ and from \cite[Lemma 3.1]{GLP17} we see that $\norm{c}^2 = \sum_i \left( \IE_{q^n} \sum_{j,k} a_{ijk} \sigma_j \sigma_k\right)^2 \le Cn^3$, so that
	\begin{align*}
		q^n(\abs{f} > t) &\le q^n\left(\abs{f-f_1} > \frac{t}{2}\right) + q^n\left(\abs{f_1} > \frac{t}{2}\right)
		\\ &\le 2 \exp\left( - \frac{t^{2/3}}{2^{2/3} C(\alpha) \norm{A}} \right) + 2 \exp\left(- \frac{t^2}{4C(\alpha) \norm{c}^2} \right) 
		\\ &\le 2 \exp\left( - \frac{t^{2/3}}{2^{2/3} C(\alpha) n} \right) + 2 \exp\left(- \frac{t^2}{4C(\alpha) C n^3}\right),
	\end{align*}
	and since $t > 2C n^{3/2}$ implies $-\frac{2^{2/3}t^{4/3}}{4Cn^2} \le -1$ we obtain
	\begin{align*}
		q^n(\abs{f} > t) \le 4 \exp\left( - \frac{t^{2/3}}{2^{2/3}C(\alpha) n} \right).
	\end{align*}
\end{proof}

Lastly, let us extend this line of thought to prove concentration of measure of polynomials of the Ising model in the presence of an external field.

\begin{proof}[Proof of Theorem \ref{Ising:Thm:ConcPolynomialsExternalField}]
	Let us prove by induction that for $p \ge 2$ we have for $f = f_{d,A}$
	\begin{equation}	\label{Ising:eqn:lpNormVsHSnorm}
		\norm{f - \IE_{q^n} f}_p^2 \le c_d p^d \norm{A}_2^2. 
	\end{equation}
	First, for $d = 1$ this is clear since $f \coloneqq f_{1,A}(X) = \sum_i a_i \tilde{X_i}$ and by equation \eqref{Ising:eqn:momentinequality} we have for $p \ge 2$
	\[
		\norm{f - \IE_{q^n} f}_p^2 \le 2Cp \norm{\mathfrak{h}f}_p^2 = 2Cp \norm{A}_2^2.
	\]
	Now for any $k$ use \eqref{Ising:eqn:momentinequality} again to get
	\begin{align*}
		\norm{f_{d,A} - \IE_{q^n} f_{d,A}}_p^2 &\le 2Cp \norm{\mathfrak{h}f}_p^2 = 2Cp \norm{\sum_{i = 1}^n (\mathfrak{h}_i f)^2}_{p/2} \le 2Cp \sum_{i = 1}^n \norm{\mathfrak{h}_i f}_p^2 \\
		&\le 2Cp \sum_{i = 1}^n c_{d-1} p^{d-1} \norm{A^{(i)}}_2^2 = 2C_d p^d \sum_{i = 1}^n \norm{A^{(i)}}_2^2 = 2C_d p^d \norm{A}_2^2.
	\end{align*}
	Here we have used the fact for any $f_{d,A}$ we have $\mathfrak{h}_i f = c_d \abs{f_{d-1, A^{(i)}} - \IE f_{d-1, A^{(i)}}}$, where $(A^{(i)})_{i_1,\ldots,i_{d-1}} = A_{i_1,\ldots,i_{d-1},i}$ is a symmetric $(d-1)$-tensor with vanishing diagonal. \par
	From equation \eqref{Ising:eqn:lpNormVsHSnorm} the first inequality easily follows as already shown in the proof of Theorem \ref{Thm:d-th-order-conc}. The second inequality is a consequence of  \[\norm{A}_2^{2/d} = \left( \sum_{i_1,\ldots, i_d} a_{i_1,\ldots,i_d}^2 \right)^{1/d} \le n \norm{A}_\infty^{2/d}. \]
\end{proof}
Note that for the case $d = 3$ and for Ising models without an external field, this translates into the previous Example, since by spin-flip symmetry we have $\IE \tilde{X_{ijk}} = \IE X_{ijk} = 0$. Additionally, for $d = 4$ we have concentration of the polynomial
\[
	f_{4,A}(X) = \sum_{ijkl} a_{ijkl}(X_{ijkl} - \IE X_{ijkl} - 6 X_{ij} \IE X_{kl} + 6 \IE X_{ij} \IE X_{kl})
\]
in absence of an external field. Here, the $6 = \binom{4}{2}$ is merely a combinatorial factor, we could also write $f_{4,A}$ in a symmetric form.

	
\subsection{Random permutations}	\label{section:Applications:randompermutations}
Next we consider random permutations which we shall describe as a probability measure on $\{1, \ldots, n\}^n$, more precisely as the uniform measure $\sigma_n$ on $S_n \coloneqq \{ (x_1, \ldots, x_n) : x_i \neq x_j \text{ for all } i \neq j \}$. With this definition it fits into our framework. \par 
Since conditioning on $n-1$ variables is useless (as the disintegrated measure will be a Dirac measure on the remaining element $x_i$ and thus a LSI cannot hold for either difference operator), we shall work with $\mathcal{I}_2 \coloneqq \{ I \subset \{1,\ldots, n\}, \abs{I} = 2 \}$. In this case, it is easy to see that for any $I = \{i,j\} \in \mathcal{I}_2$ the Markov kernel is given by $m_{\overline{x_I}} = \frac{1}{2}(\delta_{(x_i,x_j)} + \delta_{(x_j,x_i)})$, where $\{ x_i, x_j \} = \{1,\ldots,n\}\backslash \overline{x_{I}}$.  So denoting by $\tau_I \coloneqq \tau_{ij}: S_n \to S_n$ the function which switches the $i$-th and $j$-th entry, we can rewrite the difference operator as 
\[
	\partial_I f(x_1, \ldots, x_n)^2 = \frac{1}{2} \int (f(x) - f(\overline{x_I}, y_I))^2 dm_{\overline{x_I}}(y_I) = \frac{1}{4} (f(x) - f(\tau_I x))^2,
\]
and
\[
	\mathfrak{h}_I f(x)^2 = \frac{1}{2} \abs{f(x) - f(\tau_{I}x)}^2.
\]
We can rephrase \cite[Theorem 1]{LY98} in the following way.

\begin{lemma}
	Consider $(S_n, \sigma_n)$ and $\mathcal{I} = \mathcal{I}_2$. Then there exists a constant $c > 0$ independent of $n$ such that
	\[
		\Ent_{\sigma_n}(f^2) \le 2 c \frac{\log n}{n} \int \abs{\partial f}^2 d\sigma_n,
	\]
	i.e. $(S_n, \sigma_n)$ satisfies $LSI_\partial(\frac{c \log n}{n})$.
\end{lemma}
	
\begin{proof}
	The proof is rewriting the statement of \cite[Theorem 1]{LY98} in our notation, using the fact that the conditional measures are two-point Dirac measures, as follows
	\begin{align*}
		\Ent_{\sigma_n}(f^2) &\le c \log n \frac{1}{2n} \IE_{\sigma_n} \sum_{i \neq j} (f(\tau_{ij}x) - f(x))^2 \\
		&= 2c\frac{\log n}{n} \sum_{i \neq j} \int \iint (f(\overline{x_{ij}}, x_{ij}) - f(\overline{x_{ij}}, y_{ij}))^2 dm_{\overline{x_{ij}}}(x_{ij}) dm_{\overline{x_{ij}}}(y_{ij}) d\overline{\pi_{ij}}(\overline{x_{ij}}) \\
		&= 2c\frac{\log n}{n} \int \abs{\partial f}^2 d\sigma_n.
	\end{align*}
\end{proof}

The fact that the logarithmic Sobolev constant tends to zero with $n \to \infty$ is a matter of normalization. An interpretation in the context of Markov chains requires a different normalization of the difference operator, i.e. by $\abs{\mathcal{I}}^{-1}$ (see also the Remark in section \ref{section:HigherOrderDiff}), resulting in a logarithmic Sobolev constant given by $(n-1) \log n$. Moreover, this definition of a gradient has an interesting property, since for any $I = \{i,j\} \in \mathcal{I}_2$ we obtain
\[
	\mathfrak{h}_I(\mathfrak{h}_{I} f) = \abs{\mathfrak{h}_I f(x) - \mathfrak{h}_I f(\tau_{ij}x)} = \abs{\abs{f(x) - f(\tau_{ij}x)} - \abs{f(\tau_{ij}x) - f(x)}} = 0.
\]
	
\subsection{Bernoulli-Laplace and symmetric simple exclusion process}	\label{section:Applications:BLSSE}
There are two other Markov chains, whose Dirichlet form can be described in terms of a subset $\mathcal{I}$ and first-order difference operators $\partial_I$, which are the Bernoulli-Laplace model and the symmetric simple exclusion process. \par
More specifically, define on $S_n \coloneqq \{ 0,1 \}^n$ the subset known as a slice of the hypercube $C_{n,r} = \{ x \in \{ 0,1 \}^n : \sum_i x_i = r \}$, the uniform measure $\mu_{n,r}$ on $C_{n,r}$ and the two generators acting on functions on $C_{n,r}$ as
\[
	K_{n,r}f (\eta) = \sum_{i,j} \eta_i (1-\eta_j) (f(\tau_{ij} \eta) - f(\eta)),
\]
which is the generator of the so-called Bernoulli-Laplace model, and
\[
	L_{n,r}f(\eta) = \sum_{i = 1}^n (f(\tau_{i,i+1}\eta) - f(\eta)),
\]
called the symmetric simple exclusion process, where $\tau_{ij}: S_n \to S_n$ is the switching between the $i$-th and the $j$-th coordinate and we let $\tau_{n,n+1} \coloneqq \tau_{n,1}$. In \cite[Theorem 4, Theorem 5]{LY98} sharp logarithmic Sobolev constants are derived with respect to the Dirichlet form $D_{n,r}^K(f) = -\IE_{\mu_{n,r}} fL_{n,r}f$ and $D_{n,r}^L(f) = - \IE_{\mu_{n,r}} fK_{n,r}f$ (although with different normalizations), and these correspond to logarithmic Sobolev inequalities with respect to $\partial$ in the following way.

\begin{lemma}
	For $\mathcal{I} = \mathcal{I}_{2,<} = \{ (i,j) : i < j \}$ we have $\int \abs{\partial f}^2 d\mu_{n,r} = D_{n,r}^K(f)$ and for $\mathcal{I} = \mathcal{I}_1 = \{ (i,i+1) : i \in \{1, \ldots, n\} \}$ we obtain $	\int \abs{\partial f}^2 d\mu_{n,r} = D_{n,r}^L(f).$ \par
	As a consequence, $\mu_{n,r}$ satisfies a logarithmic Sobolev inequality with respect to $(\partial, \mathcal{I}_{2,<})$ with constant $c \frac{\log \frac{n^2}{r(n-r)}}{n}$ and a logarithmic Sobolev inequality with constant $cn^2$ with respect to $(\partial, \mathcal{I}_1)$, where $c$ is a constant independent of $n$ and $r$.
\end{lemma}

\begin{proof}
	Let us fix $n,r$ and drop all subscripts $n,r$, i.e. write $D^L$ for $D^L_{n,r}$, $D^K$ for $D^K_{n,r}$ and $\mu$ for $\mu_{n,r}$. For $\mathcal{I} = \mathcal{I}_{2,<}$ let $(i,j)$ be given and consider the projection $\pi_{ij}(x) = \overline{x_{ij}}$. We have 
	\[
	m_{\overline{x_{ij}}} = \begin{cases}
					\delta_{(1,1)}							& \sum_k (\overline{x_{ij}})_k = r-2 \\
	                        	\frac{1}{2}(\delta_{(0,1)} + \delta_{(1,0)}) 	& \sum_k (\overline{x_{ij}})_k = r-1 \\
	                        	\delta_{(0,0)} 							& \sum_k (\overline{x_{ij}})_k = r.
	                        \end{cases}
	\]
	and thus
	\[
		\int \abs{\partial f}^2 d\mu= \sum_{(i,j) \in \mathcal{I}_{2,<}} \int (\partial_{ij}f)^2 d\mu = \frac{1}{2} \sum_{(i,j)} \int (f(\eta) - f(\tau_{ij}\eta))^2 \eta_i (1-\eta_j) d\mu = D^K(f).
	\]
	In the second case note that $\pi_{i,i+1}(x) = \overline{x_{i,i+1}}$ (with the convention $(n,n+1) = (n,1)$) is just a special case of the $(i,j)$ above, and again we obtain
	\[
		\int \abs{\partial f}^2 d\mu = \sum_{i=1}^n \int (\partial_{i}f)^2 d\mu = \frac{1}{2} \sum_{i = 1}^n \int (f(\eta) - f(\tau_{i,i+1}\eta))^2 d\mu = D^L(f).
	\]	
	Note that we omit $\eta_i(1-\eta_{i+1})$ since otherwise we obtain $\tau_{i,i+1}\eta = \eta$. \par
	The logarithmic Sobolev inequality then follows from \cite[Theorem 4, Theorem 5]{LY98}, taking into account the missing renormalization.
\end{proof}
	
\section{Approximate tensorization of the relative entropy in product spaces}	\label{section:ApproximateTensorization}
In this section we shall reformulate and provide a complete proof of a result by K. Marton \cite{Ma15} and moreover rewrite it in the terms of entropy (of functions) instead of relative entropy of measures. To this end, let $\mathcal{X}$ be a finite set, $\mathcal{X}^n$ its $n$-fold product and fix a probability measure $q^n$ on $\mathcal{X}^n$, which does not necessarily need to be a product measure. Denote by $d_{TV}$ the total variation between two measures defined as
\[
	d_{TV}(\mu,\nu) \coloneqq \sup_{A \subset \mathcal{X}^n} \abs{\mu(A) - \nu(A)} = \frac{1}{2} \sum_{x \in \mathcal{X}^n} \abs{\mu(\{x\}) - \nu(\{x\})},
\]
and by $W_2$ the Wasserstein-$2$-type distance
\[
	W_2(\mu,\nu) \coloneqq \inf_{\pi \in C(\mu,\nu)} \left( \sum_{i = 1}^n \pi(x_i \neq y_i)^2 \right)^{1/2},
\]
where $C(\mu,\nu)$ is the set of all couplings of $\mu$ and $\nu$, i.e. probability measures $\pi$ on $\mathcal{X}^n \times \mathcal{X}^n$ with marginals $\mu$ and $\nu$. \par
Note that the infimum in the definition is always attained, since $C(\mu,\nu)$ is a compact subset of $\mathcal{P}(\mathcal{X}^n \times \mathcal{X}^n)$ equipped with the weak topology and the map $\pi \mapsto \left( \sum_{i =1}^n \pi(x_i \neq y_i)^2 \right)^{1/2}$ is lower semicontinuous. This fact and the gluing lemma for measures with a common marginal can be used to prove that $W_2$ is a distance function on $\mathcal{P}(\mathcal{X}^n)$, see for example \cite[Chapter 6]{Vil08} for a similar line of reasoning and \cite[Theorem 2.1]{AG13} for the gluing lemma. Denote by $\mu_i, \nu_i$ the pushforward measure under the projection onto the $i$-th coordinate of $\mu$ and $\nu$ respectively. By the subadditivity of the square root (for the upper bound for $W_2$) as well as the fact that every $\pi = \otimes_{i = 1}^n \pi_i$ on $\mathcal{X}^n \times \mathcal{X}^n$ of $\mu, \nu$ induces (by the projection onto the coordinates $x_i, y_i$) a coupling $\pi_i$ of $\mu_i, \nu_i$, we obtain 
\begin{equation}	\label{AT:eqn:ComparisonTVandW2}
\left( \sum_{i = 1}^n d_{TV}^2(\mu_i, \nu_i) \right)^{1/2} \le W_2(\mu,\nu) \le d_{TV}(\mu,\nu).
\end{equation}
Moreover, by $H(\mu \mid\mid \nu)$ we denote the relative entropy of $\mu$ with respect to $\nu$ given by $H(\mu \mid\mid \nu) = \int \frac{d\mu}{d\nu} \log \frac{d\mu}{d\nu} d\nu$ (whenever this exists). We will need the following lemma, which is also found in \cite[Lemma 2]{Ma15}.

\begin{lemma}\label{AT:Lem:RelEntropyIneq}
	Let $p \ll q$ be two measures and define $\beta \coloneqq \inf_{x \in \mathcal{X}_+} q(x)$, where $\mathcal{X}_+ \coloneqq \{ x \in \mathcal{X} : q(x) > 0 \}$. Then $H(p \mid\mid q) \le \min\left( \frac{2}{\beta} d_{TV}(p,q), \frac{4}{\beta} d_{TV}^2(p,q)\right)$.
	\end{lemma}

\begin{proof}
	The shifted logarithm $f(x) \coloneqq \log(1+x)$ is a concave function on $(-1,\infty)$, so that for any  $x \ge 0$ we have $f(x) \le f'(0)x = x$. Rewrite $\frac{p}{q} = 1 + \frac{p-q}{q}$ to obtain
	\begin{align*}
		H(p \mid \mid q) &= \sum_{\mathcal{X}_+} q \left(1 + \frac{p-q}{q} \right) f\left( \frac{p-q}{q} \right) \le \sum_{\mathcal{X}_+} q \left(1 + \frac{p-q}{q} \right) \frac{p- q}{q} = \sum_{\mathcal{X}_+} \frac{(p-q)^2}{q},
	\end{align*}
	where we have omitted writing the variable $x$. Lastly, using $ q(x)^{-1} < \beta^{-1}$ and $(p(x)-q(x))^2 \le \abs{p(x) - q(x)}$ we conclude
	\[
		H(p \mid \mid q) \le \frac{1}{\beta} \sum_{x \in \mathcal{X}} \abs{p(x) - s(x)} = \frac{2}{\beta} d_{TV}(p,q)
	\]
	or only $ q(x)^{-1} < \beta^{-1}$  to get 
	\[
		H(p \mid \mid q) \le \frac{1}{\beta} \left(\sum_{x \in \mathcal{X}} \abs{p(x) - q(x)}\right)^2 = \frac{4}{\beta} d_{TV}^2(p,q).
	\]
\end{proof}

We are now ready to prove the following result. We use the same notations as in the previous section, i.e. for any measure $p$ on $\mathcal{X}^n$ we denote by $p_I(\cdot \mid \overline{y_I})$ the conditional probability measure on $\mathcal{X}^I$ given by conditioning on $\overline{y_I}$.

\begin{thm}	\label{AT:Thm:ApproximateTensorization}
	Let $q^n$ be a measure with full support on $\mathcal{X}^n$.
	\begin{enumerate}[$(i)$]
		\item Let $\beta \coloneqq \min_{i =1, \ldots, n} \min_{x \in \mathcal{X}^n} q_i(x_i \mid \overline{x_i})$, $p^n$ a probability measure and assume that for all subsets $I \subset \{1,\ldots,n\}$ and all $\overline{y_I} \in \overline{\mathcal{X}^I}$ we have
		\begin{equation}	\label{AT:Thm:eqn:W2estimate}
			W_2^2( p_I(\cdot \mid \overline{y_I}), q_I(\cdot, \overline{y_I})) \le C \sum_{i \in I} \IE_{p_I(\cdot \mid \overline{y_I})} d_{TV}^2(p_i(\cdot \mid \overline{y_i}), q_i(\cdot \mid \overline{y_i})),
		\end{equation}
		then
		\begin{equation}	\label{AT:Thm:eqn:ConclusionEntropy}
			H(p^n \mid\mid q^n) \le \frac{2C}{\beta} \sum_{i = 1}^n \IE_{\overline{p_i}} H(p_i(\cdot \mid \overline{y_i}) \mid\mid q_i(\cdot \mid \overline{y_i}))
		\end{equation}
		\item If $f$ denotes the density of $p^n$ with respect to $q^n$, then this can be rewritten as
		\begin{equation}	\label{AT:Thm:eqn:ConclusionRelativeEntropy}
			\Ent_{q^n}(f) \le \frac{2C}{\beta} \sum_{i = 1}^n \int \Ent_{q_i(\cdot \mid \overline{y_i})}(f(\overline{y_i}, \cdot)) d\overline{q_i}(\overline{y_i}).
		\end{equation}
		\item Assume that the coupling matrix $A = (a_{ij})_{i \neq j}$ (see section \ref{section:Applications:Ising}) of $q^n$ satisfies the condition $\norm{A}_{2 \to 2} < 1$. Then \eqref{AT:Thm:eqn:W2estimate} holds with $C = (1-\norm{A}_{2 \to 2})^{-2}$, so that also \eqref{AT:Thm:eqn:ConclusionEntropy} and \eqref{AT:Thm:eqn:ConclusionRelativeEntropy} hold with the same constant.
	\end{enumerate}
\end{thm}

\begin{proof}First note that $\beta > 0$ due to the assumption of $q^n$ having full support. \par
	$(i)$: We will prove the theorem by induction. In the case $n = 1$ there is nothing to prove if one interprets $q_1(\cdot \mid \overline{y_1}) = q$. Using the disintegration theorem for the relative entropy (see for example \cite[Theorem D.13]{DZ10} for the formula) gives
	\begin{equation}
		\Ent_{q^n}(f) = \frac{1}{n} \sum_{i = 1}^n \Ent_{q _i}\left( \frac{dp_i}{dq_i} \right) + \int \Ent_{\overline{q_i}(\cdot \mid \overline{y_i})}\left( \frac{d\overline{p_i}(\cdot \mid y_i)}{d\overline{q_i}(\cdot \mid y_i)} \right) dp_i(y_i),
	\end{equation}
	which can be restated as
	\begin{equation}	\label{AT:Thm:eqn:decompositionEntropy}
		H(p^n \mid\mid q^n) = \frac{1}{n} \sum_{i = 1}^n H(p_i \mid \mid q_i) + \frac{1}{n} \sum_{i =1}^n\int H(\overline{p_i}(\cdot \mid y_i) \mid \mid \overline{q_i}(\cdot \mid y_i)) dp_i(y_i).
	\end{equation}
	We will treat the two terms separately. For the first term, using the estimate $H(p_i \mid\mid q_i) \le \frac{4}{\beta} d_{TV}^2(p_i, q_i)$ from Lemma \ref{AT:Lem:RelEntropyIneq}, \eqref{AT:eqn:ComparisonTVandW2}, \eqref{AT:Thm:eqn:W2estimate} and Pinsker's inequality gives
	\begin{align*}
		\frac{1}{n} \sum_{i = 1}^n H(p_i \mid \mid q_i) &\le \frac{4}{\beta n} \sum_{i =1}^n d_{TV}^2(p_i, q_i) \le \frac{4}{\beta n} W_2^2(p^n, q^n) \\ 
		&\le \frac{4C}{\beta n} \sum_{i = 1}^n \IE_{p^n} d_{TV}^2(p_i(\cdot \mid\overline{y_i}), q_i(\cdot \mid \overline{y_i})) \\
		&\le \frac{2C}{\beta n} \sum_{i = 1}^n \IE_{p^n} H(p_i(\cdot \mid \overline{y_i}) \mid \mid q_i(\cdot \mid \overline{y_i})).
	\end{align*}
	For the second term we use the induction hypothesis. For each fixed $i \in \{1,\ldots,n\}$ and $y_i \in \mathcal{X}$ we interpret $\overline{q_i}(\cdot \mid y_i)$ as a measure on $\overline{\mathcal{X}_i}$, for which
	\begin{align*}
		\beta(\overline{q_i}(\cdot \mid y_i)) &= \min_{j \neq i} \min_{x \in \overline{\mathcal{X}_i}}  \frac{\overline{q_i}(x \mid y_i)}{\overline{q_i}(z \in \overline{\mathcal{X}_i} : \overline{pr_j}(z) = \overline{x_j} \mid y_i)} \\ &=  \min_{j \neq i} \min_{x \in \overline{\mathcal{X}_i}} \frac{q^n(x, y_i)}{q^n(z \in \mathcal{X}^n: \overline{pr_j}(z) = \overline{x_j}, pr_i(z) = y_i)} \\
		&\ge \min_{i = 1,\ldots,n} \min_{y_i \in \mathcal{X}} \min_{x \in \overline{\mathcal{X}_i}} \frac{q^n(x,y_i)}{q^n(z \in \mathcal{X}^n: pr_i(z) = y_i)} \\
		&= \beta(q^n),
	\end{align*}
	and for which \eqref{AT:Thm:eqn:W2estimate} holds with the same constant $C$. To use \eqref{AT:Thm:eqn:ConclusionEntropy} let us write $y \in \overline{\mathcal{X}_i}$ for a generic vector. We need to find the conditional probability of the measure $\overline{q_i}(\cdot \mid y_i)$ with respect to the projection $\overline{pr_j}: \overline{\mathcal{X}_i} \to \overline{\mathcal{X}_{ij}}$ for some $j \neq i$. A short calculation shows that this is given by $p_j(y_j \mid \overline{y_j}, y_i)$, which is the conditional probability of $p^n$ given $\overline{pr_j} = (\overline{y_j}, y_i)$. Thus we obtain 
	\begin{align*}
		&\int H(\overline{p_i}(\cdot \mid y_i) \mid \mid \overline{q_i}(\cdot \mid y_i)) dp_i(y_i) \\
		&\le \frac{2C}{\beta} \sum_{y_i \in \mathcal{X}} p_i(y_i) \sum_{j \neq i} \sum_{\overline{y_{j}}} \frac{p^n(pr_i(z) = y_i, \overline{pr_j}(z) = \overline{y_{j}})}{p_i(y_i)} H( p_j(\cdot \mid \overline{y_{j}}, y_i) \mid \mid q_j(\cdot \mid \overline{y_{j}}, y_i)) \\
		&= \frac{2C}{\beta} \sum_{j \neq i} \IE_{p^n} H(p_j(\cdot \mid \overline{y_j}) \mid\mid q_j(\cdot \mid \overline{y_j})). 
	\end{align*}
	Summation over $i$ gives
	\[
		\frac{1}{n} \sum_{i =1}^n\int H(\overline{p_i}(\cdot \mid y_i) \mid \mid \overline{q_i}(\cdot \mid y_i)) dp_i(y_i) \le \frac{2C}{\beta}(1-1/n) \sum_{i = 1}^n \IE_{p^n} H(p_i(\cdot \mid \overline{y_i}) \mid\mid q_i(\cdot \mid \overline{y_i})),
	\]
	which combined with the first part yields the claim. \par
	$(ii)$: \eqref{AT:Thm:eqn:ConclusionRelativeEntropy} is a simple rewriting  of \eqref{AT:Thm:eqn:ConclusionEntropy}, noting that as a consequence of the disintegration theorem (or in this case Bayes' theorem) we have
	\[
		\frac{dp_i(\cdot \mid \overline{y_i})}{dq_i(\cdot \mid \overline{y_i})}(y_i) = \frac{f(\overline{y_i}, y_i)}{\int f(\overline{y_i}, x_i) dq_i(x_i \mid \overline{y_i})} 
	\]
	and $\frac{d\overline{p_i}}{d\overline{q_i}}(\overline{x_i}) = \int f(\overline{x_i}, x_i) dq_i(x_i \mid \overline{x_i})$. \par
	$(iii)$: See \cite[Theorem 2]{Ma15}.
\end{proof}

\begin{rem}
	As mentioned, in \cite[Theorem 1]{Ma15} it is stated that using the quantity 
	\[ \beta \coloneqq \inf_{i = 1,\ldots,n} \inf_{x \in \mathcal{X}^n : q^n(x) > 0} q_i(x_i \mid \overline{x_i}) \] 
	one can deduce $q^n(pr_i(x) = x_i) \ge \beta$ for all $x_i$ such that the LHS is nonzero. This is possible only if $q^n$ has full support. A counterexample is given by the push-forward of a random uniform permutation under the map $\sigma \mapsto (\sigma_1, \ldots, \sigma_n)$, which satisfies $\beta = 1$. \par 
	Another possibility would have been to modify the quantity as 
	\[ 
	\tilde{\beta}(q^n) \coloneqq \inf_{i = 1,\ldots,n} \inf_{x \in \mathcal{X}^n : q^n(x) > 0} q^n(pr_i(x) = x_i),	
	\]
	but this definition does not behave well under conditional probabilities. Indeed, it is not true that in general that for a fixed $y_i \in \mathcal{X}$ we also have $\tilde{\beta}(\overline{q_i}(\cdot \mid x_i)) \ge \tilde{\beta}(q^n)$, which can be seen in examples.
\end{rem}

As a consequence it is easy to prove a modified logarithmic Sobolev inequality under the conditions of Theorem \ref{AT:Thm:ApproximateTensorization}.

\begin{cor}
	Let $q^n$ be a measure on $\mathcal{X}^n$ with full support and assume that either \eqref{AT:Thm:eqn:W2estimate} holds or the coupling matrix $A$ satisfies $\norm{A}_{2 \to 2} < 1$ as in Theorem \ref{AT:Thm:ApproximateTensorization}$(iii)$. Then we have
	\[
		\Ent_{q^n}(e^f) \le \frac{2c}{\beta} \int \abs{\partial f}^2 e^f dq^n.
	\]
\end{cor}

\begin{proof}
	First let us note that for any probability measure $\mu$ we have
	\begin{align}	\label{AT:eqn:EntIneq}
		\Ent_{\mu}(e^f) \le \Cov_\mu(f,e^f) \le \frac{1}{2} \iint (f(x) - f(y))^2 e^{f(x)} d\mu(y) d\mu(x),
	\end{align}
	where $\Cov_\mu$ denotes the covariance under $\mu$. Indeed, this is easily seen by using Jensen's inequality to obtain $\Ent_{\mu}(e^f) \le \Cov_\mu(f,e^f)$ in combination with the elementary inequality $(a-b)(e^a - e^b) \le \frac{1}{2} (a-b)^2 (e^a + e^b)$ and the symmetry in the covariance. \par
	Now use Theorem \ref{AT:Thm:ApproximateTensorization} applied to the function $e^f$ for any function $f: \mathcal{X}^n \to \IR$, so that
	\begin{align}	\label{AT:eqn:Entef}
		\Ent_{q^n}(e^f) \le \frac{2c}{\beta} \sum_{i = 1}^n \int \Ent_{q_i(\cdot \mid \overline{y_i})}(e^{f(\cdot, \overline{y_i})}) d\overline{q_i}(\overline{y_i}),
	\end{align}
	and \eqref{AT:eqn:EntIneq} to $\mu = q_i(\cdot \mid \overline{y_i})$ to get
	\begin{align*}
		\Ent_{q^n}(e^f) \le \frac{c}{\beta} \sum_{i = 1}^n \iint (f(y) - f(\overline{y_i}, \tilde{y_i}))^2 dq_i(y_i \mid \overline{y_i}) e^{f(y)} dq^n(y) = \frac{2c}{\beta} \int \abs{\partial f}^2 e^f dq^n.
	\end{align*}
\end{proof}

In the notation of \cite{BG99} (see also \cite{GS16}) it means that the difference operator $\abs{\partial{f}}$ satisfies a modified logarithmic Sobolev inequality with constant $\frac{4c}{\beta}$. Thus, by \cite[Theorem 2.1]{BG99} (or more specifically, the remark thereafter) this yields Gaussian tail behavior with variance $\frac{4c}{\beta}$ for any (probabilistic) ``Lipschitz function'' $f$, i.e. for any function $f$ such that $\abs{\partial f} \le 1$.

\printbibliography
\end{document}